\documentclass[a4paper,11pt]{amsart}

\usepackage{amsfonts, latexsym, amsmath, amssymb, amsthm, amscd}
\usepackage[all]{xy}
\usepackage[english]{babel}
\usepackage[utf8]{inputenc}
\usepackage{graphicx}
\usepackage{booktabs}
\usepackage{array, tabularx}
\usepackage{textcomp}
\usepackage{tikz, tikz-cd}
\usepackage{multirow}
\usepackage{paralist}
\usepackage{comment}
\usepackage{url}
\usepackage{tensor}
\usepackage{MnSymbol}
\usepackage[hypertexnames=false,
backref=page,
    pdftex,
    pdfpagemode=UseNone,
    breaklinks=true,
    extension=pdf,
    colorlinks=true,
    linkcolor=blue,
    citecolor=blue,
    urlcolor=blue,
]{hyperref}

\usepackage{enumitem}
\usepackage[top=1.5in, bottom=.9in, left=0.9in, right=0.9in]{geometry}

\newcommand{\referenza}{}

\newtheorem{thm}{Theorem}[section]
\newtheorem*{thm*}{Theorem \referenza}

\newtheorem*{cor*}{Corollary \referenza}
\newtheorem{lem}[thm]{Lemma}
\newtheorem*{lem*}{Lemma \referenza}

\newtheorem{prop}[thm]{Proposition}
\newtheorem*{prop*}{Proposition \referenza}

\newtheorem*{conj*}{Conjecture \referenza}
\theoremstyle{definition}

\newtheorem*{rmk*}{Remark}

\theoremstyle{definition}

\newtheorem{defi}[thm]{Definition}

\def\frg{{\mathfrak{g}}}

\def\db{{\bar{\partial}}}

\DeclareMathOperator{\Span}{Span}
\DeclareMathOperator{\id}{id}
\DeclareMathOperator{\End}{End}
\DeclareMathOperator{\Ker}{Ker}

\DeclareMathOperator{\vol}{vol}

\DeclareMathOperator{\im}{Im}
\DeclareMathSymbol{\Finv} {\mathord}{AMSb}{"60}

\newcommand\restrict[1]{\raisebox{-.5ex}{$|$}_{#1}}

\newcommand{\C}{\mathbb{C}}
\newcommand{\Z}{\mathbb{Z}}
\newcommand{\del}{\partial}
\newcommand{\delbar}{\overline{\partial}}

\numberwithin{equation}{section}

\let\phi\varphi

\DeclareFontFamily{U}{MnSymbolC}{}
\DeclareSymbolFont{MnSyC}{U}{MnSymbolC}{m}{n}
\DeclareFontShape{U}{MnSymbolC}{m}{n}{
    <-6>  MnSymbolC5
   <6-7>  MnSymbolC6
   <7-8>  MnSymbolC7
   <8-9>  MnSymbolC8
   <9-10> MnSymbolC9
  <10-12> MnSymbolC10
  <12->   MnSymbolC12}{}
\DeclareMathSymbol{\intprod}{\mathbin}{MnSyC}{'270}

\allowdisplaybreaks

\author{Tommaso Sferruzza}
\address[Tommaso Sferruzza]{
Dipartimento di Scienze Matematiche, Fisiche e Informatiche\\
Unità di Mate\-matica e Informatica\\
Università degli studi di Parma}
\email{tommaso.sferruzza@unipr.it}

\author{Adriano Tomassini}
\address[Adriano Tomassini]{
Dipartimento di Scienze Matematiche, Fisiche e Informatiche\\
Unità di Mate\-matica e Informatica\\
Università degli studi di Parma}
\email{adriano.tomassini@unpir.it}

\title{Dolbeault and Bott-Chern formalities: deformations and $\del\delbar$-lemma}

\thanks{The first author has been supported by GNSAGA of INdAM. The second author has been supported by the project PRIN2017 “Real and Complex Manifolds: Topology, Geometry and holomorphic dynamics” (code 2017JZ2SW5), and by GNSAGA of INdAM.}

\date{\today}

\begin{document}

\begin{abstract}
It is proved that the properties of being Dolbeault formal and geometrically-Bott-Chern-formal are not closed under holomorphic deformations of the complex structure. Further, we construct a compact complex manifold which satisfies the $\del\delbar$-lemma but admits a non vanishing Aeppli-Bott-Chern-Massey product.
\end{abstract}

\maketitle
\section{Introduction}

Let $M$ be a compact manifold. As a consequence of formality theory of Sullivan, in \cite{DGMS} it is proved that if $M$ admits a K\"ahler structure, namely an integrable almost complex structure $J$ and a Hermitian metric $g$ whose fundamental form $\omega$ is closed, then the de Rham complex of $M$ is formal as a differential graded algebra; in particular, all Massey products on $M$ vanish. More generally, the same holds if $M$ satisfies the $\del\delbar$-lemma, for example, when $M$ is a Moishezon or a Fujiki class $\mathcal{C}$ manifold. Furhermore, the celebrated theorem by Kodaira and Spencer states that on a compact complex manifold the K\"ahler condition is stable under small deformations of the complex structure. A stability result also holds for small deformations of compact complex manifolds satisfying the $\del\delbar$-lemma, as proved in \cite{Wu}, see also \cite{AT12}. Equivalently, on a compact complex manifold, the existence of a K\"ahler metric, respectively, the validity of the $\del\delbar$-lemma, is a open condition under small deformations of the complex structure.

Motivated by Sullivan \cite{Sull}, in \cite{Kot} Kotschick defined \emph{metrically formal metrics} on a  manifold as Riemannian metrics such that the space of harmonic forms has a structure of algebra induced by the wedge product of forms, i.e., the wedge product of harmonic forms is still harmonic. Accordingly, a closed manifold is called \emph{geometrically formal} if it admits a metrically formal metric.

In the complex setting, Neisendorfer and Taylor introduced in \cite{NT} a notion of formality. More specifically, a complex manifold is said to be \emph{Dolbeault formal} if its double complex of complex differential forms is formal as a bidifferential bigraded algebra. In the same work, it is proved that Dolbeault formality provides an obstruction for the $\del\delbar$-lemma, i.e., a compact complex manifold satisfying the $\del\delbar$-lemma is Dolbeault formal.

Inspired by Kotschick, in \cite{TT}, respectively \cite{AT15}, the authors introduced the notions of \emph{geometrically Dolbeault formal manifolds}, respectively \emph{geometrically-Bott-Chern-formal manifolds}, as compact complex manifolds admitting a Hermitian metric $g$ such that the space of Dolbeault, respectively Bott-Chern, harmonic forms is an algebra. Accordingly, \emph{Dolbeault-Massey products} and \emph{Aeppli-Bott-Chern-Massey products} are defined in \cite{TT} and \cite{AT15}, in which it is shown that they provide an obstruction to, respectively, geometrical Dolbeault formality and geometrical Bott-Chern formality.
In particular, in \cite{TT} it is proved that Dolbeault formality and geometrical Dolbeault formality are not stable under small deformations of the complex structure. The same result holds for Bott-Chern-geometrical formality, as proved in \cite{TTo}.

In the present paper, we are interested in studying the closedness properties under small deformations of the complex structure of Dolbeault formality, geometrical Dolbeault formality and geometrical Bott-Chern formality, and in the relation between $\del\delbar$-lemma and geometrical Bott-Chern formality. By definition, a property $P$ depending on the complex structure of a complex manifold $(M,J)$ is said to be \emph{closed under holomorphic deformations of the complex structure} if, for every holomorphic family $\{(M,J_t)\}_{t\in\Delta}$, with $\Delta=\{z\in\C:|z|<1\}\subset\C$ and $(M_0,J_0)=(M,J)$,
\[
\text{if}\,\, P\,\, \text{holds on}\,\, (M,J_t),\,\,\text{for every}\,\, t\in\Delta\setminus \{0\}\,\,\,\Rightarrow\,\,\, P \,\,\text{holds on}\,\, (M,J).
\]
The first result regards the non closedness property of Dolbeault and Bott-Chern formalities; we prove the following, see  Theorem \ref{thm:dolb_non_closedness} and Theorem \ref{thm:bott_non_closedness}.
\begin{thm}\label{thm:main_1}
(1) The property of being geometrically Dolbeault formal, Dolbeault formal, weakly Dolbeault formal, and the vanishing of Dolbeault-Massey triple products are not closed under holomorphic deformations.\\
(2) The property of being geometrically-Bott-Chern-formal and the vanishing of Aeppli-Bott-Chern-Massey triple products are not closed under holomorphic deformations.
\end{thm}
As already remarked, a compact complex manifold satisfying the $\del\delbar$-lemma is formal in the sense of Sullivan and Dolbeault formal; in particular, the Massey products, respectively Dolbeault-Massey products, vanish. In contrast, Aeppli-Bott-Chern-Massey products do not provide an obstruction for the validity of the $\del\delbar$-lemma. Indeed, we prove the following, see Theorem \ref{thm:ABC_deldelbar}.
\begin{thm}\label{thm:main_2}
There exists a compact complex manifold satisfying the $\del\delbar$-lemma and admitting a non-vanishing $ABC$-Massey triple product.
\end{thm}
We prove statement (1) of Theorem \ref{thm:main_1} by constructing a holomorphic family of compact complex manifolds $\{M_t\}_{t\in\Delta}$ obtained as a deformation of the complex structure of the holomorphically parallelizable Nakamura manifold, such that each $M_{t}$ is geometrically Dolbeault formal and Dolbeault formal for $t\in\Delta\setminus\{0\}$, but $M_0$ has a non vanishing Dolbeault-Massey triple product. This will assure that on each $M_t$ every triple Dolbeault-Massey product is vanishing, for $t\in\Delta\setminus\{0\}$, but $M_0$ is neither Dolbeault formal, nor geometrically Dolbeault formal.

To prove statement (2) of Theorem \ref{thm:main_1}, we use a different presentation of the holomorphically parallelizable Nakamura manifold selecting a suitable family of lattices. Then, we consider a holomorphic deformation of the complex structure such that each $M_t$ is geometrically-Bott-Chern-formal, for $t\neq 0$, but $M_0$ has a non vanishing $ABC$-Massey product, hence on each $M_t$ every $ABC$-Massey product vanishes but $M_0$ is not geometrically-Bott-Chern-formal.

In order to prove Theorem \ref{thm:main_2}, we start by constructing a complex orbifold obtained as a quotient of the Iwasawa manifold and by showing that it satisfies the $\del\delbar$-lemma and it admits a non vanishing $ABC$-Massey product. Then, we explicitly construct a smooth resolution $\tilde{M}$ of such orbifold and we conclude the proof by showing that $\tilde{M}$ still admits a non vanishing $ABC$-Massey product and it still satisfies the $\del\delbar$-lemma.

The paper is organised as follows. In section \ref{sec-notations}, we set the notations and recall the main facts of Hodge theory and cohomologies of compact complex manifolds. In section \ref{sec:form}, we recollect the definitions of Dolbeault formality and weak Dolbeault formality from the point of view of bidifferential bigraded algebras and we recollect the definitions of geometrically Dolbeault formal manifolds (respectively, geometrically-Bott-Chern-formal manifolds) and Dolbeault-Massey products (respectively, Aeppli-Bott-Chern-Massey products) as introduced in \cite{TT}, respectively \cite{TTo}. Section \ref{sec:orb} briefly gathers the main facts about cohomologies of complex orbifolds proved classically in \cite{Bai56,Joy00,SA56} and more recently in \cite{A13orb,ASTT,Stel,Stel2}. Finally, sections \ref{sec-Dolbeault}, \ref{sec-Bott_Chern}, and \ref{sec:deldelbar_ABC} are devoted to the proofs of Theorem \ref{thm:dolb_non_closedness}, Theorem \ref{thm:bott_non_closedness}, and Theorem \ref{thm:ABC_deldelbar}.
\vspace{1.0cm}

{\it Acknowledgement.} The authors kindly thank Luis Ugarte for useful discussions and suggestions in constructing the families of Theorem \ref{thm:dolb_non_closedness}. They also would like to thank Daniele Angella, Andrea Cattaneo, and Nicoletta Tardini for valuable comments and remarks which helped in the presentation of the results. Special thanks are due to Jonas Stelzig for many useful discussions.

\section{Preliminaries}\label{sec-notations}
Let $(M,J,g,\omega)$ be a compact Hermitian manifold with $\dim_{\C}M=n$, i.e., a compact complex manifold $(M,J)$, where $J\in\End(TM)$ such that $J^2=-id_{TM}$ is the integrable almost-complex structure on $M$ and $g$ is a Hermitian metric on $(M,J)$, i.e., a Riemannian metric on $M$ such that $J$ is an isometry with respect to $g$. The associated fundamental form $\omega\in\bigwedge^2M$ of $g$ is defined by the expression $\omega(\cdot,\cdot)=g(J\cdot,\cdot)$.

Once we extend $J$ to the complexified cotangent bundle $(T_{\C}M)^*$ and its exterior powers $\bigwedge_{\C}^kM$, we can decompose such bundles in terms of the $\pm i$-eigenspaces of $J$, which we denote by $(T^{1,0}M)^{\ast}$ and $ (T^{0,1}M)^*$, as follows
\begin{equation*}
(T_{\C}M)^*=(T^{1,0}M)^{\ast}\oplus (T^{0,1}M)^*,
\qquad
\textstyle\bigwedge_{\C}^{k}M=\oplus_{p+q=k}\textstyle\bigwedge^{p,q}M,
\end{equation*}
where each bundle $\bigwedge^{p,q}M:=\bigwedge^{p}(T^{1,0}M)^*\otimes\bigwedge^q(T^{0,1}M)^*$ is the bundle of \emph{$(p,q)$-forms} on $M$. We will denote its $\mathcal{C}^{\infty}$ global sections $\Gamma(M,\bigwedge^{p,q}M)$ by $\mathcal{A}^{p,q}M$.

With respect to such decompositions, the exterior differential $d$ acting at the level of $(p,q)$-forms on $M$ 
splits as $d=\del+\delbar$, where
\begin{equation}
\del\restrict{\mathcal{A}^{p,q}M}:=\pi^{p+1,q}(d(\mathcal{A}^{p,q}M)),\qquad \delbar\restrict{\mathcal{A}^{p,q}M}:=\pi^{p,q+1}(d(\mathcal{A}^{p,q}M)),
\end{equation}
are the projections of $d(\mathcal{A}^{p,q}M)$ onto, respectively, $\mathcal{A}^{p+1,q}M$ and $\mathcal{A}^{p,q+1}M$.

We will consider on $(M,J,g,\omega)$ the $\C$-antilinear Hodge $\ast$-operator with respect to $g$, i.e., the operator
\begin{gather*}
\ast\colon \mathcal{A}^{p,q}(M)\rightarrow\mathcal{A}^{n-p,n-q}(M)\\
\beta\mapsto \ast\beta
\end{gather*}
defined by $\alpha\wedge \ast\beta:=g(\alpha,\beta)\vol_g$, for any $\alpha\in\mathcal{A}^{p,q}M$, where we use the symbol $g$ for the $\C$-antilinear extension of $g$ to any $\mathcal{A}^{p,q}(M)$, and $\vol_g$ is the volume form naturally associated to $\omega$.

With respect to the $L^2(M)$-product on $\mathcal{A}^{p,q}M$
\begin{equation*}
(\alpha,\beta):=\int_{M}\alpha\wedge\ast\beta,
\end{equation*}
for any $\alpha$, $\beta\in\mathcal{A}^{p,q}M$, we consider the adjoint operators $\del^*$ and $\delbar^*$ of, respectively, $\del$ and $\delbar$, which can be written as
\begin{equation*}
\del^*=-\ast\del\ast,\qquad
\delbar^*=-\ast\delbar\ast.
\end{equation*}
Then, the \emph{Dolbeault Laplacian}, the \emph{Bott-Chern Laplacian}, and the \emph{Aeppli Laplacian} are defined as
\begin{align*}
&\Delta_{\delbar}:=\delbar\delbar^*+\delbar^*\delbar,\\
&\Delta_{BC}:=\del\delbar\delbar^*\del^*+\delbar^*\del^*\del\delbar+\delbar^*\del\del^*\delbar+\del^*\delbar\delbar^*\del+\delbar^*\delbar+\del^*\del,\\
&\Delta_{A}:=\del\del^*+\delbar\delbar^*+\delbar^*\del^*\del\delbar+\del\delbar\delbar^*\del^*+\del\delbar^*\delbar\del^*+\delbar\del^*\del\delbar^*.
\end{align*}
We note that $\Delta_{\delbar}$ is a self-adjoint elliptic second-order differential operator, whereas $\Delta_{BC}$ and $\Delta_{A}$ are self-adjoint elliptic fourth-order operators.

Let us now recall the definitions for the spaces of \emph{Dolbeault}, \emph{Bott-Chern}, and \emph{Aeppli cohomologies}, namely
\begin{equation*}
H_{\delbar}^{\bullet,\bullet}(M)=\frac{\Ker\delbar}{\im \delbar},\qquad H_{BC}^{\bullet,\bullet}(M)=\frac{\Ker\del\cap\Ker\delbar}{\im\del\delbar},\qquad H_A^{\bullet,\bullet}(M)=\frac{\Ker\del\delbar}{\im\del +\im\delbar}
\end{equation*}
and let
\begin{equation*}
\mathcal{H}_{\delbar}^{\bullet,\bullet}(M,g)=\Ker\Delta_{\delbar},\qquad \mathcal{H}_{BC}^{\bullet,\bullet}(M,g)=\Ker\Delta_{BC},\qquad \mathcal{H}_{A}^{\bullet,\bullet}(M,g)=\Ker\Delta_A
\end{equation*}
be the spaces of Dolbeault-, Bott-Chern-, Aeppli-harmonics forms on $M$ with respect to $g$, also denoted, respectively, as $\Delta_{\delbar}$-, $\Delta_{BC}$-, and $\Delta_A$-\emph{harmonic forms}.

Note that, since $M$ is compact, $\alpha\in\mathcal{A}^{p,q}(M)$ is \begin{itemize}
\item $\Delta_{\delbar}$-harmonic if and only if
\begin{equation}\label{eq:harm_delbar_forms}
\Delta_{\delbar}\alpha=0\Leftrightarrow\begin{cases}
\delbar\alpha=0\\
\delbar\ast\alpha=0,
\end{cases}
\end{equation}
\item $\Delta_{BC}$-harmonic if and only if
\begin{equation}\label{eq:harm_BC_forms}
\Delta_{BC}\alpha=0\Leftrightarrow\begin{cases}
\del\alpha=0\\
\delbar\alpha=0\\
\del\delbar\ast\alpha=0,
\end{cases}
\end{equation}
\item $\Delta_A$-harmonic if and only if
\begin{equation}\label{eq:harm_A_forms}
\Delta_{A}\alpha=0\Leftrightarrow\begin{cases}
\del\ast\alpha=0\\
\delbar\ast\alpha=0\\
\del\delbar\alpha=0.
\end{cases}
\end{equation}
\end{itemize}
The Hodge $\ast$-operator induces isomorphisms both at the level of cohomology and harmonic representatives between Bott-Chern- and Aeppli-cohomology, i.e.,
\begin{equation*}
\ast \left(H_{BC}^{p,q}(M)\right)\simeq H_A^{n-p,n-q}(M),\quad \ast\left(\mathcal{H}_{BC}^{p,q}(M,g)\right)\simeq\mathcal{H}_A^{n-p,n-q}(M,g).
\end{equation*}
Analogously to the classical Hodge theory, the natural injections
\begin{equation*}
\mathcal{H}_{\square}^{p,q}(M,g)\hookrightarrow H_{\square}^{p,q}(M), \qquad\text{for}\quad \square\in\{\delbar,BC,A\},
\end{equation*}
are in fact $\C$-linear isomorphisms of complex vector spaces.

Finally, we recall that a compact complex manifold $(M,J)$ is said to satisfy the \emph{$\del\delbar$-lemma} if
\[
\Ker\del\cap\Ker\delbar\cap(\im\del+\im\delbar)=\im\del\delbar.
\] 
\section{Obstructions to complex formalities}\label{sec:form}
In the real setting, the notion of formality has been introduced by Sullivan as follows: a differentiable manifold $M$ is said to be \emph{formal in the sense of Sullivan} if its de Rham complex $(\bigwedge^{\bullet}M,d)$ is equivalent, as a differential graded algebra, to a differential graded algebra with zero differential $(\mathcal{B},d\equiv 0)$, see \cite{Su} for further details.

Starting from this idea, similar notions have been introduced in the complex setting by Neisendorfer and Taylor in \cite{NT}.
We recall the following definitions.

\begin{defi}

A \emph{differential bi-graded algebra} (shortly, DBA) is a bi-graded commutative algebra $\mathcal{A}=\oplus_{i,j} \mathcal{A}_{i,j}$, endowed with a differential $\delbar$ of type $(0,1)$, i.e., $\delbar(\mathcal{A}_{i.j})\subset(\mathcal{A}_{i,j+1})$, which is a derivation, i.e., $\delbar(\alpha\cdot\beta)=\delbar\alpha\cdot\beta +(-1)^{\deg \alpha}\alpha\cdot\delbar\beta$.

A \emph{bidifferential bigraded algebra} (shortly, BBA) is a DBA $(\mathcal{A}=\oplus_{i,j}\mathcal{A}_{i,j},\delbar)$ further endowed with a differential $\del$ of type $(1,0)$, i.e., $\del \mathcal{A}_{i,j}\subset \mathcal{A}_{i+1,j}$, which is a derivation and anti-commutes with $\delbar$, i.e., $\del\delbar=-\delbar\del.$

Morphisms between DBA's and BBA's, are bi-degree preserving morphisms of algebras which commute with the differential of DBA's (or differentials, for BBA's).
\end{defi}
In particular, the cohomology of a BBA $(\mathcal{A},\delbar_{\mathcal{A}},\del_{\mathcal{A}})$ with respect to $\delbar_{\mathcal{A}}$ can be defined by setting $H_{\delbar}(\mathcal{A}):=\frac{\Ker\delbar_{\mathcal{A}}}{\im\delbar_{\mathcal{A}}}$. We note that given a BBA $(\mathcal{A},\delbar_{\mathcal{A}},\del_{\mathcal{A}})$, also $(H_{\delbar}(\mathcal{A}),0,\del_{\mathcal{A}})$ is a BBA.

Now, given  any morphism of BBA's $f\colon (\mathcal{A},\delbar_{\mathcal{A}},\del_{\mathcal{A}})\rightarrow (\mathcal{B},\delbar_{\mathcal{B}},\del_{\mathcal{B}})$, it commutes with differentials and, therefore, it induces a well-defined morphism at the level of cohomology
\begin{gather*}
H_{\delbar}(f)\colon (H_{\delbar}(\mathcal{A}),0,\del_{\mathcal{A}})\rightarrow (H_{\delbar}(\mathcal{B}),0,\del_{\mathcal{B}})\\
[\alpha]_{\delbar_{\mathcal{A}}}\mapsto [f(\alpha)]_{\delbar_{\mathcal{B}}}.
\end{gather*}
We define in similar manner the cohomology of a DBA's $(\mathcal{A},\delbar_{\mathcal{A}})$ with respect to the differential $\delbar_{\mathcal{A}}$ as $H_{\delbar}(\mathcal{A}):=\frac{\Ker\delbar_{\mathcal{A}}}{\im\delbar_{\mathcal{A}}}$. Also in this case, $(H_{\delbar}(\mathcal{A}),0)$ is DBA and morphisms of DBA's induce morphisms in cohomology.
\begin{defi}
Two BBA's $(\mathcal{A},\delbar_{\mathcal{A}},\del_{\mathcal{A}})$ and $(\mathcal{B},\delbar_{\mathcal{B}},\del_{\mathcal{B}})$ are said to be equivalent if there exists a family of BBA's $\{(\mathcal{C}_i,\del_{\mathcal{C}_i},\delbar_{\mathcal{C}_i})\}_{i=0}^{2k+1}$  such that $(\mathcal{C}_0,\delbar_{\mathcal{C}_0},\del_{\mathcal{C}_0})=(\mathcal{A},\delbar_{\mathcal{A}},\del_{\mathcal{A}})$ and $(\mathcal{C}_{2k+1},\delbar_{\mathcal{C}_{2k+1}},\del_{\mathcal{C}_{2k+1}})=(\mathcal{B},\delbar_{\mathcal{B}},\del_{\mathcal{B}})$ and morphisms of BBA's $f_i$ and $g_i$
\begin{equation*}
\begin{tikzcd}
& (\mathcal{C}_{2i+1},\delbar_{\mathcal{C}_{2i+1}},\del_{\mathcal{C}_{2i+1}}) \arrow[swap]{dl}{f_i} \arrow{dr}{g_i} & \\
(\mathcal{C}_{2i},\delbar_{\mathcal{C}_{2i}},\del_{\mathcal{C}_{2i}}) & & \quad(\mathcal{C}_{2i+2},\delbar_{\mathcal{C}_{2i+2}},\del_{\mathcal{C}_{2i+2}})
\end{tikzcd}
\end{equation*}
such that $f_i$ and $g_i$ induce, respectively, isomorphisms $H_{\delbar}(f_i)$ and $H_{\delbar}(g_i)$ in cohomology.
\end{defi}
We can now give the definitions of formality as done in \cite{NT}.
\begin{defi}
A complex manifold $(M,J)$ is said to be \emph{Dolbeault formal} if the Dolbeault complex $(\bigwedge^{\bullet,\bullet}M,\delbar,\del)$ is equivalent, as a BBA, to a BBA $(B,0,\del)$ whose differential $\delbar$ is zero. 
\end{defi}
A weaker condition arises when we consider the complex of the $(p,q)$-forms endowed only with the usual $\delbar$ differential. \begin{defi}
A complex manifold $(M,J)$ is said to be \emph{weakly Dolbeault formal} if $(\bigwedge^{\bullet,\bullet}(M),\delbar)$ is equivalent, as a DBA, to a DBA $(B,0)$ whose differential $\delbar$ is zero.
\end{defi}
It is easy to see that Dolbeault formality implies weak Dolbeault formality, since the operator $\del$ does not come into play in the definition of weak Dolbeault formality. Moreover, any manifold satisfying the $\del\delbar$-lemma (in particular, compact K\"ahler manifolds) are Dolbeault formal, see \cite{NT}.

In the wake of the definition of being geometrically formal according to Kotschick for differentiable manifolds, see \cite{Kot}, a notion of formality has been defined as well in the complex setting, see \cite{TT}. More precisely, we have the following.
\begin{defi}
A compact complex manifold $(M,J)$ is said to be \emph{geometrically Dolbeault formal} if $(M,J)$ admits a Hermitian metric $g$ such that $\mathcal{H}_{\delbar}^{\bullet,\bullet}(M,g)$ have a structure of algebra with respect to the wedge product.
\end{defi}
The following relations hold, as proved in \cite[Proposition 2.1, Proposition 2.2]{TT}.
\begin{prop}\label{prop:geom_form}
Let $(M,J)$ be a compact complex manifold. Then, if $(M,J,g,\omega)$ is geometrically Dolbeault formal,
\begin{enumerate}
\item $(M,J)$ is also weakly Dolbeault formal;
\item if $\mathcal{H}_{\delbar}^{\bullet,\bullet}(M,g)$ is $\del$-invariant, then $(M,J)$ is Dolbeault formal.
\end{enumerate}
\end{prop}
As Massey products yield obstructions to formality in the sense of Sullivan, the idea translates well in the complex setting. In fact, let $(M,J)$ be a compact complex manifold, and let
\begin{equation*}
\mathfrak{a}=[\alpha]\in H_{\delbar}^{p,q}(M),\qquad \mathfrak{b}=[\beta]\in H_{\delbar}^{r,s}(M),\qquad \mathfrak{c}=[\gamma]\in H_{\delbar}^{u,v}(M),
\end{equation*}
such that $\mathfrak{a}\smile\mathfrak{b}=0$ in $H_{\delbar}^{p+r,q+s}(M)$, $\mathfrak{b}\smile\mathfrak{c}=0$ in $H_{\delbar}^{r+u,s+v}(M)$, i.e.,
\begin{equation*}
\alpha\wedge\beta=\delbar f_{\alpha\beta},\qquad \beta\wedge\gamma=\delbar f_{\beta\gamma},
\end{equation*}
for $f_{\alpha\beta}\in\mathcal{A}^{p+r,q+s-1}M$, $f_{\beta\gamma}\in\mathcal{A}^{r+u,s+v-1}M$. The \emph{Dolbeault-Massey triple product} between the classes $\mathfrak{a}$, $\mathfrak{b}$, $\mathfrak{c}$, is the following equivalence class of cohomology classes
\begin{align*}
\langle\mathfrak{a},\mathfrak{b},\mathfrak{c}\rangle_{\delbar}:=[f_{\alpha\beta}\wedge\gamma+(-1)^{p+q}\alpha\wedge f_{\beta\gamma}]\in 
\frac{H_{\delbar}^{p+r+u,q+s+v-1}(M)}{\mathfrak{c}\smile H_{\delbar}^{p+r,q+s-1}(M)+\mathfrak{a}\smile H_{\delbar}^{r+u,s+v-1}(M)}.
\end{align*}
Then, we have the following obstruction result, see \cite[Proposition 3.1, Proposition 3.2]{TT}.
\begin{prop}\label{thm:dolb-mass}
Let $(M,J)$ be a complex compact manifold. If $(M,J)$ is either  Dolbeault formal or weakly Dolbeault formal, then every
Dolbeault–Massey triple products on $(M,J)$ identically vanish.
\end{prop}
We then have the following chain of implications for compact complex manifolds
\begin{equation}\label{Dolb_imp1}
\text{$\del\delbar$-lemma}\Rightarrow\text{Dolbeault formality}\Rightarrow\text{weak Dolbeault formality}\Rightarrow\text{vanishing of $\delbar$-Massey products.}
\end{equation}
As for a Bott-Chern formality, it turns out that there is not a clear way of defining it in terms of the algebraic approach described in the first part of this paragraph, since Bott-Chern cohomology not defined as the cohomology of a complex. However, as introduced in \cite{AT15}, a notion of formality can still be defined.
\begin{defi}
A compact complex manifold $(M,J)$ is said to be \emph{geometrically-Bott-Chern-formal} if $(M,J)$ admits a Hermitian metric $g$ such that $\mathcal{H}_{BC}^{\bullet,\bullet}(M,g)$ has a structure of algebra with respect to the wedge product.  
\end{defi}
Also, it is possible to define an adapted version of Massey triple products, as follows. Let $(M,J)$ be a compact complex manifold and let
\begin{equation*}
\mathfrak{a}=[\alpha]\in H_{BC}^{p,q}(M),\qquad \mathfrak{b}=[\beta]\in H_{BC}^{r,s}(M), \qquad \mathfrak{c}=[\gamma]\in H_{BC}^{u,v}(M),
\end{equation*}
such that $\mathfrak{a}\smile\mathfrak{b}=0$ in $H_{BC}^{p+r,q+s}(M)$ and $\mathfrak{b}\smile\mathfrak{c}=0$ in $H_{BC}^{r+u,s+v}(M)$, i.e.,
\begin{equation*}
(-1)^{p+q}\alpha\wedge\beta=\del\delbar g_{\alpha\beta}, \qquad (-1)^{r+s}\beta\wedge\gamma=\del\delbar g_{\beta\gamma},
\end{equation*}
for $g_{\alpha\beta}\in\mathcal{A}^{p+r-1,q+s-1}M$, $g_{\beta\gamma}\in\mathcal{A}^{r+u-1,s+v-1}M$. The \emph{Aeppli-Bott-Chern Massey triple product} $\langle\mathfrak{a},\mathfrak{b},\mathfrak{c}\rangle_{ABC}$  of the cohomology classes $\mathfrak{a}$, $\mathfrak{b}$, $\mathfrak{c}$, is the following equivalence class of cohomology classes
\begin{align}\label{def:ABCM}
[(-1)^{p+q}\alpha\wedge g_{\beta\gamma}-(-1)^{r+s}g_{\alpha\beta}\wedge\gamma]\in
\frac{H_A^{p+r+u-1,q+s+v-1}(M)}{[\alpha]_{BC}\smile H_A^{r+u-1,s+v-1}(M)+[\gamma]_{BC}\smile H_A^{p+r-1,q+s-1}(M)}.
\end{align}
Aeppli-Bott-Chern Massey triple products yield an obstruction to being geometrically-Bott-Chern-formal, as follows, see \cite[Theorem 3.2]{TTo}.
\begin{prop}\label{prop:gf_ABC}
Let $(M,J)$ be a compact complex manifold. If $(M,J,g,\omega)$ is geometrically-Bott-Chern-formal, then
every Aeppli–Bott–Chern Massey triple products is trivial.
\end{prop}

\section{Cohomologies of complex orbifolds}\label{sec:orb}
In this section, we briefly recall the main facts about complex orbifolds and their cohomologies. In particular, we will focus on complex orbifolds of global-quotient-type, a class of simple yet useful orbifolds, of which we will make use of in section \ref{sec:deldelbar_ABC}.

We start by recalling the definition of a complex orbifold, as originally introduced by \cite{SA56,Joy00}.
\begin{defi}
A complex space of dimension $n$ is said to be a \emph{complex orbifold of complex dimension} $n$ if it is singular and its singularities are locally isomorphic to quotient singularities $\C^n/ G$, where $G$ is a finite subgroup of $GL(n;\C)$.
\end{defi}

Tensors on a complex orbifold $\hat{M}$, such as vector fields, differential forms, or metrics, are defined to be, locally at a point $p\in\hat{M}$, as tensors which are $G$-invariant on $\C^n$, with $G\in GL(n;\C)$ such that, locally at $p$, it holds that $\hat{M}\simeq \C^n/G.$

Let us then consider the graded complex of complex forms on the complex orbifold $\hat{M}$, namely,  $(\bigwedge_{\C}^{\bullet}\hat{M},d)$, and its associated bigraded complex $(\bigwedge^{\bullet,\bullet}\hat{M},\delbar,\del)$. As recalled in section \ref{sec-notations} for the usual cohomologies of manifolds, we can define \emph{de Rham}, \emph{Dolbeault}, \emph{Bott-Chern}, and \emph{Aeppli orbifold cohomologies} as 
\begin{align}\label{eq:coom_orb}
&\textstyle H_{dR}^{p,q}(\hat{M})=\frac{\Ker d}{\im d}\cap \bigwedge^{p,q}(\hat{M}),&\quad &H_{\delbar}^{p,q}(\hat{M})=\frac{\Ker\delbar}{\im\delbar}\cap\textstyle\bigwedge^{p,q}(\hat{M}),\\
&\textstyle H_{BC}^{p,q}(\hat{M})=\frac{\Ker\del\cap\Ker\delbar}{\im\del\delbar}\cap\bigwedge^{p,q}(\hat{M}),&\quad &H_A^{p,q}(\hat{M})=\frac{\Ker\del\delbar}{\im\del+\im\delbar}\cap\textstyle\bigwedge^{p,q}(\hat{M}).\label{eq:coom_orb1}
\end{align}
Starting from the complexes $(\bigwedge^{\bullet}\hat{M},d)$ and $(\bigwedge^{\bullet,\bullet}\hat{M},\del,\delbar)$, a spectral sequence $\{(E_r^{\bullet},d_r)\}$ can be defined, so that $E_1^{\bullet}\simeq H_{\delbar}^{\bullet,\bullet}(\hat{M})$. From such sequence, known as \emph{Hodge and Fr\"olicher spectral sequence} of $\hat{M}$, one can derive the Fr\"olicher inequality 
\begin{equation}\label{eq:fr-orb}
\sum_{p+q=k} \dim_{\C}H_{\delbar}^{p,q}(\hat{M})\geq \dim_{\C}H_{dR}^k(\hat{M};\C).
\end{equation}
A complex orbifold is said to \emph{satisfy the $\del\delbar$-lemma} if the natural map $H_{BC}^{p,q}(\hat{M})\rightarrow H_{\delbar}^{p,q}(\hat{M})$ is injective. Among many other characterizations, such property is equivalent, for a complex orbifold, to equality holding in equation (\ref{eq:fr-orb}) and to have isomorphisms induced by conjugation in Dolbeault cohomology, i.e.,
\begin{equation}
\overline{H_{\delbar}^{p,q}(\hat{M})}\simeq H_{\delbar}^{q,p}(\hat{M}),
\end{equation}
see \cite{DGMS}.

Once we fix an Hermitian metric $g$ on a compact complex orbifold $\hat{M}$ of complex dimension $n$, one can define the $\C$-antilinear Hodge $\ast$-operator
\[\textstyle
\ast\colon\bigwedge^{p,q}\hat{M}\rightarrow\bigwedge^{n-p,n-q}\hat{M},
\]
the operators
\[
d^*=-\ast d \ast, \quad \del^*=-\ast\del\ast, \quad \delbar^*=-\ast\delbar\ast,
\]
the de Rham Laplacians $\Delta$, Dolbeault Laplacian $\Delta_{\delbar}$, Bott-Chern Laplacian $\hat{\Delta}_{BC}$, and Aeppli Laplacian 
$\hat{\Delta}_A$ and their kernels
\begin{gather*}
\mathcal{H}_{\sharp}^k=\{\alpha\in\textstyle\bigwedge^k\hat{M}:\Delta\alpha=0\},\\
\mathcal{H}^{p,q}_{\sharp}=\{\alpha\in\textstyle\bigwedge^{p,q}\hat{M}: \Delta_{\sharp}\alpha=0\}, \quad\text{for}\quad \sharp\in\{\delbar,BC,A\}.
\end{gather*}
Harmonic forms on $\hat{M}$ with respect to each Laplacian can be characterized as section \ref{sec-notations} in equations (\ref{eq:harm_delbar_forms}), (\ref{eq:harm_BC_forms}), and (\ref{eq:harm_A_forms}).

For a compact complex orbifold, the following theorem holds, see \cite[Theorem 1]{SA56},\cite[Theorem K]{Bai56}.
\begin{thm}\label{thm:coom_orb1}
Let $\hat{M}$ be a compact complex orbifold of complex dimension $n$ and $g$ an Hermitian metric on $\hat{M}$. The following isomorphisms hold
\begin{align*}
H_{dR}^k(\hat{M};\C)&\rightarrow\mathcal{H}_{dR}^k(\hat{M},g)\\
H_{\delbar}^{p,q}(\hat{M})&\rightarrow\mathcal{H}_{\delbar}^{p,q}(\hat{M},g).
\end{align*}
Moreover, the Hodge $\ast$-operator yields, respectively, the isomorphisms
\begin{align*}
H_{dR}^k(\hat{M},\C)\simeq H_{dR}^{2n-k}(\hat{M},\C)\\
H_{\delbar}^{p,q}(\hat{M})\simeq H_{\delbar}^{n-p,n-q}(\hat{M}).
\end{align*}
\end{thm}

Let us now consider the following class of complex orbifolds.
\begin{defi}
A complex orbifold $\hat{M}$ is said to be of \emph{global-quotient-type} if $\hat{M}=M/ G$, where $M$ is a complex manifold and $G$ is finite subgroup of the group of biholomorphisms of $M$.
\end{defi}
For compact orbifolds of global-quotient-type, besides Theorem \ref{thm:coom_orb1}, also Bott-Chern and Aeppli cohomologies can be computed in terms of harmonic representatives, as in the following theorem.
\begin{thm}\label{thm:coom_orb2}
Let $\hat{M}$ be a compact complex orbifold of global-quotient type and $g$ a Hermitian metric on $\hat{M}$. Then, the following isomorphisms hold
\begin{align*}
H_{BC}^{p,q}(\hat{M})&\rightarrow \mathcal{H}_{BC}^{p,q}(\hat{M},g)\\
H_A^{p,q}(\hat{M})&\rightarrow \mathcal{H}_A^{p,q}(\hat{M},g).
\end{align*}
In particular, the Hodge $\ast$-operator induces the isomorphisms
\begin{equation}
H_{BC}^{p,q}(\hat{M})\simeq H_A^{n-p,n-q}(\hat{M}).
\end{equation}
\end{thm}
We conclude this section by recalling the property of the pull-back map of a proper surjective morphism of compact complex orbifolds, see \cite{A13orb}.
\begin{thm}\label{thm:pullback_orb}
Let $\hat{M}$ and $\hat{N}$ be compact complex orbifolds of the same complex dimension, and let $\pi\colon \hat{M}\rightarrow\hat{N}$ be a proper surjective morphism of complex orbifolds. Then the map $\pi\colon \hat{M}\rightarrow \hat{N}$ induces injective morphisms
\begin{align*}
\pi_{dR}^*\colon H_{dR}^k(\hat{N})&\rightarrow H_{dR}^k(\hat{M})\\
\pi_{\delbar}^*\colon H_{\delbar}^{p,q}(\hat{N})&\rightarrow H_{\delbar}^{p,q}(\hat{M})\\
\pi_{BC}^{\ast}\colon H_{BC}^{p,q}(\hat{N})&\rightarrow H_{BC}^{p,q}(\hat{M}).
\end{align*} 
\end{thm}
\section{Dolbeault formalities are not closed}\label{sec-Dolbeault}
In this section we state and prove the non closedness result for the Dolbeault formalities as defined in section \ref{sec:form}.

We recall that, by definition of a property closed under holomorphic deformations, for our purposes it will suffice to show the existence of a holomorphic family of compact complex manifolds $\{M_t\}_{t\in\Delta}$, $\Delta=\{z\in\C:|z|<1\}\subset\C$, such that each $M_{t}$ is geometrically Dolbeault formal and Dolbeault formal for $t\in\Delta\setminus\{0\}$, but $M_0$ has a non vanishing Dolbeault-Massey triple product. Given Proposition \ref{prop:geom_form} and the chain of implications (\ref{Dolb_imp1}), this will assure that each $M_t$ is also weakly-Dolbeault formal and every triple Dolbeault-Massey product is vanishing, for $t\in\Delta\setminus\{0\}$, but $M_0$ is neither weakly-Dolbeault formal, Dolbeault formal, nor geometrically Dolbeault formal, yielding the following result.

\begin{thm}\label{thm:dolb_non_closedness}
The property of being geometrically Dolbeault formal, Dolbeault formal, weakly Dolbeault formal, and the vanishing of Dolbeault-Massey triple products are not closed under holomorphic deformations.
\end{thm}
In order to prove Theorem \ref{thm:dolb_non_closedness}, we will provide a family $\{Z_t\}_{t\in\Delta}$ of holomorphic deformations of the holomorphically parallelizable Nakamura manifold such that $Z_t$ is geometrically Dolbeault formal and Dolbeault formal, for $t\neq 0$, but $Z_0$ has a non-trivial Dolbeault-Massey triple product.

To this purpose, let us start by considering
the 6-dimensional simply-connected solvable Lie group $G$, with Lie algebra $\frg$ defined by the following structure equations of the frame $\{e^1,e^2,e^3,e^4,e^5,e^6\}$ of $\frg^*$
\begin{equation}\label{ecus-g}
\left\{
\begin{array}{rl}
\!\!\!&\!\!\! de^1 = e^{16}-e^{25}, \\[5pt]
\!\!\!&\!\!\! de^2 = e^{15}+e^{26}, \\[5pt]
\!\!\!&\!\!\! de^3 = -e^{36}+e^{45}, \\[5pt]
\!\!\!&\!\!\! de^4 = -e^{35}-e^{46}, \\[5pt]
\!\!\!&\!\!\! de^5 = de^6 = 0.
\end{array}
\right.
\end{equation}
We then consider the holomorphically parallelizable complex structure on $g^*$. Define the almost complex structure on $\frg^*$, which we will denote by $J_{(0,0)}$, by setting
$$
J_{(0,0)} e^1= -e^2, \ \ J_{(0,0)} e^2=e^1, \ \ J_{(0,0)} e^3=-e^4, \ \ J_{(0,0)} e^4=e^3, \ \ J_{(0,0)} e^5=e^6, \ \ J_{(0,0)} e^6=-e^5.
$$
Therefore, the following complex forms
$$
\eta_{(0,0)}^1 = e^1+i\, e^2,  \quad\quad  \eta_{(0,0)}^2 = e^3+i\,e^4,  \quad\quad  \eta_{(0,0)}^3 = \frac12 (e^5- i\,e^6),
$$
form a basis of (1,0)-forms for $(\frg^{\ast})^{1,0}$ whose complex structure equations are
\begin{equation}\label{solv-00}
d\eta_{(0,0)}^1 = 2i\,\eta_{(0,0)}^{13},\quad\quad
d\eta_{(0,0)}^2 = -2i\,\eta_{(0,0)}^{23},\quad\quad
d\eta_{(0,0)}^3 = 0.
\end{equation}
If we fix any lattice of maximal rank in $G$, we obtain a complex $3$-dimensional solvmanifold.

We note that the relations between the choice of lattices of the Nakamura holomorphically parallelizable manifold and the dimensions of the Dolbeault and Bott-Chern cohomologies have been studied, for example, in \cite{Nakamura,AKGlob}.
In particular, for every lattice $\Gamma$ of maximal rank, the $(0,1)$-form $\eta_{(0,0)}^{\bar 3}$ defines a non-zero Dolbeault cohomology class on the compact complex manifold $M_{(0,0)}=(\Gamma\backslash G,J_{(0,0)})$.
Hence, we can use the class $[\eta_{(0,0)}^{\bar 3}] \in H^{0,1}_{\db}(X_{(0,0)})$ to construct an appropriate holomorphic family of deformations.

Let $\mathbf{B}=\mathbb{C}\times \Delta(0,1) \subset \mathbb{C}^2$, where $\Delta(0,1)=\{ z\in \mathbb{C}\mid |z|<1 \}$.
For any $\mathbf{t}=(t_1,t_2) \in \mathbf{B}$, we set
\begin{equation}\label{deform}
\eta_\mathbf{t}^1:=\eta_{(0,0)}^1+t_1\, \eta_{(0,0)}^{\bar{3}},\ \ \eta_\mathbf{t}^2:=\eta_{(0,0)}^2,\ \ \eta_\mathbf{t}^3:=\eta_{(0,0)}^3+ t_2\, \eta_{(0,0)}^{\bar{3}}.
\end{equation}
A direct computation shows that the structure equations of the $(1,0)$-forms $\{\eta_\mathbf{t}^1,\eta_\mathbf{t}^2,\eta_\mathbf{t}^3\}$ are
\begin{equation}\label{solv-t1t2}
\left\{
\begin{array}{rcl}
d\eta_\mathbf{t}^1 \!\!&=&\!\! \frac{2i}{1-|t_2|^2}\,\eta_\mathbf{t}^{13}-\frac{2it_2}{1-|t_2|^2}\,\eta_\mathbf{t}^{1\bar{3}}+\frac{2it_1}{1-|t_2|^2}\,\eta_\mathbf{t}^{3\bar{3}},\\[8pt]
d\eta_\mathbf{t}^2 \!\!&=&\!\! -\frac{2i}{1-|t_2|^2}\,\eta_\mathbf{t}^{23}+\frac{2it_2}{1-|t_2|^2}\,\eta_\mathbf{t}^{2\bar{3}},\\[8pt]
d\eta_\mathbf{t}^3 \!\!&=&\!\! 0.
\end{array}
\right.
\end{equation}

Therefore, for any $\mathbf{t}=(t_1,t_2) \in \mathbf{B}$, we have a left-invariant complex structure $J_\mathbf{t}$ on $\Gamma\backslash G$,
and so a compact complex manifold $M_{\bf{t}}=(\Gamma\backslash G,J_\mathbf{t})$ of complex dimension 3.


Before proceeding, we need the following result.

\begin{lem}\label{t1-nonzero}
If $t_1\not=0$ and $t_2=0$, then the compact complex manifold $X_{\mathbf{t}=(t_1,0)}$ has a non-vanishing Dolbeault-Massey triple product.
\end{lem}

\begin{proof}[Proof of Lemma \ref{t1-nonzero}]
Let us consider the Dolbeault cohomology classes $[\eta_{(t_1,0)}^3] \in H^{1,0}_{\db}(X_{(t_1,0)})$ and $[\eta_{(t_1,0)}^{\bar 3}] \in H^{0,1}_{\db}(X_{(t_1,0)})$.
From \eqref{solv-t1t2} for $t_2=0$ and $t_1\not=0$, we have the following relations:
$$
\eta_{(t_1,0)}^3 \wedge \eta_{(t_1,0)}^3 = 0, \quad\quad
\eta_{(t_1,0)}^3 \wedge \eta_{(t_1,0)}^{\bar 3} = \db \left( \frac{-i}{2 t_1} \eta_{(t_1,0)}^1 \right).
$$
Hence, $\langle [\eta_{(t_1,0)}^3],[\eta_{(t_1,0)}^3],[\eta_{(t_1,0)}^{\bar 3}] \rangle$ is a Dolbeault-Massey triple product
which is represented (up to a constant) by the (2,0)-form $\eta_{(t_1,0)}^1 \wedge \eta_{(t_1,0)}^3$.
This (2,0)-form obviously defines a non-zero Dolbeault cohomology class in $H^{2,0}_{\db}(X_{(t_1,0)})$.
Now, for showing that the product is non-trivial, it remains to prove that the class $[\eta_{(t_1,0)}^1 \wedge \eta_{(t_1,0)}^3]$
does not belong to the ideal $[\eta_{(t_1,0)}^3] \cdot H^{1,0}_{\db}(X_{(t_1,0)})$.

Suppose that $[\eta_{(t_1,0)}^1 \wedge \eta_{(t_1,0)}^3] \in [\eta_{(t_1,0)}^3] \cdot H^{1,0}_{\db}(X_{(t_1,0)})$. Then, there exists a (1,0)-form
$\alpha$ on the manifold $X_{(t_1,0)}$ satisfying $\db\alpha=0$ and $\eta_{(t_1,0)}^1 \wedge \eta_{(t_1,0)}^3=\alpha \wedge \eta_{(t_1,0)}^3$.
Now, since the complex structure is left-invariant, we can apply the symmetrization process (it preserves the bidegree of the forms) to
get an invariant (1,0)-form $\tilde\alpha$ which is $\db$-closed and satisfies
$(\eta_{(t_1,0)}^1 -\tilde\alpha) \wedge \eta_{(t_1,0)}^3=0$.
But from \eqref{solv-t1t2} for $t_2=0$ and $t_1\not=0$, it follows that
$\tilde\alpha=\lambda\, \eta_{(t_1,0)}^2+ \mu\, \eta_{(t_1,0)}^3$ for some constants $\lambda,\mu \in \mathbb{C}$ in order to be $\db$-closed,
so the condition $(\eta_{(t_1,0)}^1 -\tilde\alpha) \wedge \eta_{(t_1,0)}^3=0$ cannot be satisfied.
\end{proof}

\begin{proof}[Proof of Theorem \ref{thm:dolb_non_closedness}]
Let us now fix any $t_1^0 \in \mathbb{C}\setminus \{0\}$. For any $t_2 \in \Delta(0,1)=\{ z\in \mathbb{C}\mid |z|<1 \}$,
we consider the left-invariant complex structure $J_{\mathbf{t}=(t_1^0,t_2)}$ on $G$.
By~\eqref{solv-t1t2} the complex structure equations are
\begin{equation}\label{solv-t10-t2}
\left\{
\begin{array}{rcl}
d\eta_\mathbf{t}^1 \!\!&=&\!\! \frac{2i}{1-|t_2|^2}\,\eta_\mathbf{t}^{13}-\frac{2it_2}{1-|t_2|^2}\,\eta_\mathbf{t}^{1\bar{3}}+\frac{2it_1^0}{1-|t_2|^2}\,\eta_\mathbf{t}^{3\bar{3}},\\[8pt]
d\eta_\mathbf{t}^2 \!\!&=&\!\! -\frac{2i}{1-|t_2|^2}\,\eta_\mathbf{t}^{23}+\frac{2it_2}{1-|t_2|^2}\,\eta_\mathbf{t}^{2\bar{3}},\\[8pt]
d\eta_\mathbf{t}^3 \!\!&=&\!\! 0.
\end{array}
\right.
\end{equation}
If we take any $t_2 \in \Delta(0,1)\setminus\{0\}$, we consider the basis $\{\tau_\mathbf{t}^1,\tau_\mathbf{t}^2,\tau_\mathbf{t}^3\}$ of $(1,0)$-forms
with respect to  $J_{\mathbf{t}}$ defined by
\[
\tau_\mathbf{t}^1:= 2i\, \eta_\mathbf{t}^3, \quad
\tau_\mathbf{t}^2:=\eta_\mathbf{t}^1 + \frac{t_1^0}{t_2} \, \eta_\mathbf{t}^3, \quad
\tau_\mathbf{t}^3:=\eta_\mathbf{t}^2.
\]
It is easy to check with respect to this basis, the complex structure equations become
\begin{equation}\label{solv-t2}
\left\{
\begin{array}{rcl}
d\tau_\mathbf{t}^1 \!\!&=&\!\! 0,\\[8pt]
d\tau_\mathbf{t}^2 \!\!&=&\!\! -\frac{1}{1-|t_2|^2}\,\tau_\mathbf{t}^{12}+\frac{t_2}{1-|t_2|^2}\,\tau_\mathbf{t}^{2\bar{1}},\\[8pt]
d\tau_\mathbf{t}^3 \!\!&=&\!\! \frac{1}{1-|t_2|^2}\,\tau_\mathbf{t}^{13}-\frac{t_2}{1-|t_2|^2}\,\tau_\mathbf{t}^{3\bar{1}}.
\end{array}
\right.
\end{equation}

In \cite{AK} it is proved that there is a family of lattices $\{\Gamma_{t_2}\}_{t_2\in \Delta}$ on the Lie group $G$ such that the compact
manifold $\Gamma_{t_2}\backslash G$ endowed with the complex structure $\{J_{(t_1^0,t_2)}\}_{t_2\in\Delta}$ given by \eqref{solv-t2}
satisfies the $\partial\db$-lemma for any $t_2 \in \Delta(0,1)\setminus\{0\}$, and, therefore, is Dolbeault formal. Indeed, notice that the equations \eqref{solv-t2} are precisely the complex equations found in \cite[Table 3]{AK} for the holomorphic deformation $(C_1)$ in \cite[Proposition 4.2]{AK}.

Also, it is easy to check that the harmonic representatives of Dolbeault cohomology listed in \cite[Table 3]{AK} with respect to the canonical metric have a structure of algebra with respect to $\wedge$, therefore $M_{\bf{t}}$ is also geometrically Dolbeault formal.

Hence, we consider the following holomorphic family of compact complex manifolds $\{Z_t\}_{t\in \Delta}$.
Let us fix any $t_1^0 \in \mathbb{C}\setminus\{0\}$ and consider $t=t_2$ for $t \in \Delta(0,1)$. We take the previous lattices $\Gamma_t:=\Gamma_{t_2}$ on the Lie group $G$ given in \cite{AK} and the (left-invariant) complex structure
$J_t=J_{(t_1^0,t)}$ on $G$, to obtain the family of compact complex manifolds $\{Z_t\}=\{\Gamma_t\backslash G, J_t\}$.

As we pointed out above, each compact complex manifold $Z_t$
is Dolbeault formal and geometrically Dolbeault formal for any $t\not= 0$. However, the central fiber $Z_0$
has a non-vanishing Dolbeault-Massey triple product by Lemma \ref{t1-nonzero}, since this result
holds for any lattice of maximal rank in $G$, in particular for the given lattice $\Gamma$.
\end{proof}

\section{Bott-Chern formality is not closed}\label{sec-Bott_Chern}
In this section, we prove the non closedness result for geometrically-Bott-Chern-formal manifolds and the vanishing of Aeppli-Bott-Chern-Massey products.

As for Dolbeault formality in section \ref{sec-Dolbeault}, it suffices to show the existence of a holomorphic family of compact complex manifolds $\{M_t\}_{t\in\Delta}$, $\Delta=\{z\in\C:|z|<1\}$, such that $M_t$ is geometrically-Bott-Chern formal for $t\in\Delta\setminus\{0\}$, but $M_0$ admits a non-vanishing Aeppli-Bott-Chern-Massey triple product. In fact, by Proposition \ref{prop:gf_ABC}, $M_t$ is geometrically-Bott-Chern formal and also has no non-vanishing Aeppli-Bott-Chern-Massey triple products, whereas $M_0$ would be not geometrically-Bott-Chern formal, thus proving the following result.
\begin{thm}\label{thm:bott_non_closedness}
The property of being geometrically-Bott-Chern-formal and the vanishing of Aeppli-Bott-Chern-Massey triple products are not closed under holomorphic deformations.\end{thm}
In order to prove Theorem \ref{thm:bott_non_closedness}, we will use a different representation of the Nakamura holomorphically parallelizable manifolds by choosing a suitable family of lattices. 

Let $(M=\Gamma\backslash G,J)$ be the Nakamura holomorphically parallelizable manifold, where
\begin{itemize}
\item $G:=\C\ltimes\C^2$ is the solvable complex Lie group defined by $\gamma(z_1)*(z_2,z_3)=(e^{-z_1}z_2,e^{z_1}z_3)$;
\item $\Gamma:=(a\Z+2\pi\Z)\ltimes \Gamma''$ is a lattice of $G$ of maximal rank, with $\Gamma''$ a lattice of $\C^2$;
\item $J$ is the holomorphically parallelizable complex structure on $M$ induced by the natural standard complex structure on $\C^3\simeq\C\ltimes \C^2$.
\end{itemize}
In particular, we point out that with this choice of $\Gamma$, it holds that $h_{\delbar}^{0,1}(M,J)=3$ (see \cite{AKGlob}) and a basis of invariant $(1,0)$-forms is given by $\{\eta^1:=dz^1,\eta^2:=e^{-z_1}dz^2,\eta^3:=e^{z_1}dz^3\}$ whose structure equations are
\begin{equation}\label{eq:struct_nak}
d\eta^1=0,\quad
d\eta^2=-\eta^{12},\quad
d\eta^3=\eta^{13}.
\end{equation}
Since $[\eta^{\overline{1}}]\in H_{\delbar}^{0,1}(M)$ is a non-zero cohomology class, we can consider the deformation constructed in \cite{AK} given by the $(0,1)$-vector form $\phi(t)$ as follows
\begin{equation*}
\phi(t):=t\frac{\del}{\del z^1}\otimes \overline{\eta}^1, \qquad t\in\Delta.
\end{equation*}
The resulting almost-complex structure $J_t$ is then characterized by the following coframe of $(1,0)$-forms on $(M,J_t)$
\begin{equation*}
\begin{cases}
\eta_t^1:=\eta^1+t\overline{\eta}^1\\
\eta_t^2:=\eta^2\\
\eta_t^3:=\eta^3,
\end{cases}
\end{equation*}
whose structure equations are
\begin{equation}\label{Nak_t}
\left\{
\begin{array}{rcl}
d\eta_t^1 \!\!&=&\!\! 0,\\[8pt]
d\eta_t^2 \!\!&=&\!\! -\frac{1}{1-|t|^2}\,\eta_t^{12}+\frac{t}{1-|t|^2}\,\eta_t^{2\bar{1}},\\[8pt]
d\eta_t^3 \!\!&=&\!\! \frac{1}{1-|t|^2}\,\eta_t^{13}-\frac{t}{1-|t|^2}\,\eta_t^{3\bar{1}}.
\end{array}
\right.
\end{equation}
It is clear that $J_t$ is integrable, thus giving rise to the holomorphic family of compact complex manifolds $(M,J_t)$, for every $t\in\Delta.$

Let us fix on $M_t$ the Hermitian metric $g_t$ whose fundamental form is $\omega_t=\frac{i}{2}(\eta_t^{1\overline{1}}+\eta_t^{2\overline{2}}+\eta_t^{3\overline{3}})$. Then, as proved in \cite{AK}, for every $t\neq0$, the manifold $(M,J_t)$ satisfies the $\del\delbar$-lemma and the harmonic representatives of the Bott-Chern cohomology of $(M,J_t)$ for $t\neq0$ are as in Table \ref{tab:BC_cohom}.

It is easy to check that $\mathcal{H}_{BC}^{\bullet,\bullet}(M,g_t)$ has a structure of algebra induced by the $\wedge$ product of forms. Therefore, the manifolds $(M,J_t)$ are all geometrically Bott-Chern formal for $t\neq 0$.

\begin{proof}[Proof of Theorem \ref{thm:bott_non_closedness}]
It will suffices to construct a non zero Aeppli-Bott-Chern Massey triple product on $(M,J_0)=(M,J)$.

As proved in \cite{AKGlob}, the harmonic representatives of the Bott-Chern cohomology of $(M,J)$ with respect to the canonical diagonal metric $g$ are as listed in Table \ref{tab:BC_cohom_nak}.

As a first remark, we notice that $\mathcal{H}_{BC}^{\bullet,\bullet}(M,g)$ does not have a structure of algebra induced by the $\wedge$ product of form. In fact, the product $\eta^{12}\wedge\left( e^{\overline{z}_1-z_1}\eta^{3\overline{1}}\right)$ is not harmonic with respect to the Bott-Chern Laplacian, since
\[
e^{\overline{z}_1-z_1}\eta^{123\overline{1}}=\del\delbar(-e^{\overline{z}_1-z_1}\eta^{23}).
\]
Therefore, take the following Bott-Chern cohomology classes
\begin{equation}
\mathfrak{a}:=[\eta^{12}]_{BC}, \quad \mathfrak{b}=[e^{\overline{z}_1-z_1}\eta^{3\overline{1}}]_{BC}, \quad \mathfrak{c}:=[\eta^{\overline{12}}]_{BC}.
\end{equation}
Since $\mathfrak{a}\cup \mathfrak{b}=[e^{\overline{z}_1-z_1}\eta^{123\overline{1}}]=0\in H_{BC}^{3,1}(M)$ and clearly $\mathfrak{b}\cup \mathfrak{c}=0\in H_{BC}^{1,3}(M)$, by definition \ref{def:ABCM} we obtain that
\begin{equation}
[e^{\overline{z}_1-z_1}\eta^{23\overline{12}}]_A\in\frac{H_A^{2,2}(M)}{[\eta^{12}]_{BC}\cup H_A^{0,2}(M) + [\eta^{\overline{12}}]_{BC}\cup H_A^{2,0}(M)}
\end{equation}
is the Aeppli-Bott-Chern-Massey triple product $\langle \mathfrak{a},\mathfrak{b},\mathfrak{c}\rangle_{ABC}$.

We proceed by showing that, as a cohomology class, $[e^{\overline{z}_1-z_1}\eta^{23\overline{12}}]_A\neq 0$.
Indeed, it can be easily seen from structure equations (\ref{eq:struct_nak}) that the form $e^{\overline{z}_1-z_1}\eta^{23\overline{12}}$ is $\del\delbar$-closed and, since $\ast \left(e^{\overline{z}_1-z_1}\eta^{23\overline{12}}\right)=e^{z_1-\overline{z}_1}\eta^{1\overline{3}}$, it is a light matter of computations to show that
\begin{equation*}
\del\ast \left(e^{\overline{z}_1-z_1}\eta^{23\overline{12}}\right)=0, \quad \delbar\ast \left(e^{\overline{z}_1-z_1}\eta^{23\overline{12}}\right)=0.
\end{equation*}
Therefore, conditions (\ref{eq:harm_A_forms}) assure that $e^{\overline{z}_1-z_1}\eta^{23\overline{12}}$ is $\Delta_A$-harmonic and therefore, as a Aeppli cohomology class, $[e^{\overline{z}_1-z_1}\eta^{23\overline{12}}]_A\neq 0$.
Actually, from Table \ref{tab:BC_cohom_nak}, one can directly compute the spaces $H_A^{2,2}(M)$, $H_A^{2,0}(M)$, and $H_A^{0,2}(M)$, by the relations $H_A^{p,q}(M)=\ast\left(H_{BC}^{n-p,n-q}(M) \right)$, obtaining
\begin{align*}
H_A^{2,0}(M)&=\C\langle [e^{z_1-\overline{z}_1}\eta^{12}],[e^{\overline{z}_1-z_1}\eta^{13}],[\eta^{23}] \rangle,\\
H_A^{2,2}(M)&=\C\langle[e^{z_1-\overline{z}_1}\eta^{12\overline{13}}],[e^{z_1-\overline{z}_1}\eta^{12\overline{23}}],[e^{\overline{z}_1-z_1}\eta^{13\overline{12}}],[e^{\overline{z}_1-z_1}\eta^{13\overline{23}}],[e^{\overline{z}_1-z_1}\eta^{23\overline{12}}],[e^{z_1-\overline{z}_1}\eta^{23\overline{13}}],[\eta^{23\overline{23}}]\rangle,\\
H_A^{0,2}(M)&=\C\langle[e^{\overline{z}_1-z_1}\eta^{\overline{12}}],[e^{z_1-\overline{z}_1}\eta^{\overline{13}}],[e^{\overline{23}}] \rangle,
\end{align*}
in which we displayed the $\Delta_A$-harmonic representatives with respect to the canonical diagonal metric $g$ on $(M,J)$.

It remains to show that $[e^{\overline{z}_1-z_1}\eta^{23\overline{12}}]_A\notin [\eta^{12}]_{BC}\cup H_A^{0,2}(M) + [\eta^{\overline{12}}]_{BC}\cup H_A^{2,0}(M)$.

We point out that a generic element $\mathfrak{d}\in\eta^{12}]_{BC}\cup H_A^{0,2}(M) + [\eta^{\overline{12}}]_{BC}\cup H_A^{2,0}(M)$ can be written as
\begin{equation*}
\mathfrak{d}= [Ae^{\overline{z}_1-z_1}\eta^{12\overline{12}}+B e^{z_1-\overline{z}_1}\eta^{12\overline{13}}+C\eta^{12\overline{23}}A'e^{z_1-\overline{z}_1}\eta^{13\overline{13}}+B'e^{\overline{z}_1-z_1}\eta^{13\overline{12}}+C'\eta^{23\overline{12}}]_A,
\end{equation*}
for $A,B,C,A',B',C'\in\C$.

By contradiction, let us suppose that
\begin{equation*}
[e^{\overline{z}_1-z_1}\eta^{23\overline{12}}]_A= [Ae^{\overline{z}_1-z_1}\eta^{12\overline{12}}+B e^{z_1-\overline{z}_1}\eta^{12\overline{13}}+C\eta^{12\overline{23}}A'e^{z_1-\overline{z}_1}\eta^{13\overline{13}}+B'e^{\overline{z}_1-z_1}\eta^{13\overline{12}}+C'\eta^{23\overline{12}}]_A,
\end{equation*}
for some $A,B,C,A',B',C'\in\C$, or equivalently, by definition of Aeppli cohomology, that
\begin{equation}\label{eq:nak_example}
e^{\overline{z}_1-z_1}\eta^{23\overline{12}}=Ae^{\overline{z}_1-z_1}\eta^{12\overline{12}}+B e^{z_1-\overline{z}_1}\eta^{12\overline{13}}+C\eta^{12\overline{23}}A'e^{z_1-\overline{z}_1}\eta^{13\overline{13}}+B'e^{\overline{z}_1-z_1}\eta^{13\overline{12}}+C'\eta^{23\overline{12}}+\del\lambda+\delbar\mu,
\end{equation}
for some forms $\lambda\in\mathcal{A}^{1,2}(M)$, $\mu\in\mathcal{A}^{2,1}(M)$.

However, we observe that the following forms are $\del$ or $\delbar$ exact, i.e.,
\begin{align*}
\begin{cases}
&\eta^{12\overline{23}}=\del(\eta^{2\overline{23}})\\
&e^{\overline{z}_1-z_1}\eta^{12\overline{12}}=\del(-\frac{1}{2}e^{\overline{z}_1-z_1}\eta^{2\overline{12}})\\
&\eta^{23\overline{12}}=\delbar(-\eta^{23\overline{2}})\\
&e^{z_1-\overline{z}_1}\eta^{13\overline{13}}=\delbar(-\frac{1}{2}e^{z_1-\overline{z}_1}\eta^{13\overline{3}}),
\end{cases}
\end{align*}
therefore equation (\ref{eq:nak_example}) reduces to
\begin{equation}\label{eq:nak_example2}
e^{\overline{z}_1-z_1}\eta^{23\overline{12}}=B e^{z_1-\overline{z}_1}\eta^{12\overline{13}}+B'e^{\overline{z}_1-z_1}\eta^{13\overline{12}}+\del\lambda+\delbar\mu.
\end{equation}
In particular, since $e^{\overline{z}_1-z_1}\eta^{23\overline{12}}, e^{z_1-\overline{z}_1}\eta^{12\overline{13}}, e^{\overline{z}_1-z_1}\eta^{13\overline{12}}\in\mathcal{H}_{A}^{2,2}(M,g)$, it must hold that
\begin{equation*}
e^{\overline{z}_1-z_1}\eta^{23\overline{12}}-(B e^{z_1-\overline{z}_1}\eta^{12\overline{13}}+B'e^{\overline{z}_1-z_1}\eta^{13\overline{12}})\in\mathcal{H}_A^{2,2}(M,g)
\end{equation*}
therefore, equation (\ref{eq:nak_example2}) boils down to
\begin{equation*}
e^{\overline{z}_1-z_1}\eta^{23\overline{12}}-B e^{z_1-\overline{z}_1}\eta^{12\overline{13}}-B'e^{\overline{z}_1-z_1}\eta^{13\overline{12}}=0,
\end{equation*}
for some $B,B'\in\C$, but this clearly cannot hold. Thus, we obtain a contradiction and hence
\begin{equation*}
[e^{\overline{z}_1-z_1}\eta^{23\overline{12}}]_A\notin [\eta^{12}]_{BC}\cup H_A^{0,2}(M) + [\eta^{\overline{12}}]_{BC}\cup H_A^{2,0}(M),
\end{equation*}
showing that $\langle\mathfrak{a},\mathfrak{b},\mathfrak{c}\rangle$ defines a non vanishing Aeppli-Bott-Chern-Massey triple product on $(M,J)$.

By Proposition \ref{prop:gf_ABC}, we can conclude that $(M,J)$ is also not geometrically-Bott-Chern formal.
\end{proof}

\section{Aeppli-Bott-Chern-Massey products and the $\del\delbar$-lemma}\label{sec:deldelbar_ABC}
In this section, we show that the Aeppli-Bott-Chern-Massey triple products are not an obstruction for the $\del\delbar$-lemma on a compact complex manifold, unlike Massey triple products and Dolbeault-Massey triple products, see Theorem \ref{thm:dolb-mass}. In fact, we will costruct a global-quotient-type complex orbifold by taking the quotient of the Iwasawa manifold with respect to the action of a finite group of biholomorphisms and we will prove that it satisfies  the $\del\delbar$-lemma but it admits a non-vanishing Aeppli-Bott-Chern-Massey triple product. As a final step, we will we will construct a smooth resolution of such complex orbifold still satisfying the $\del\delbar$-lemma and admitting a non-vanishing Aeppli-Bott-Chern-Massey triple product.

\begin{thm}\label{thm:ABC_deldelbar}
There exists a compact complex manifold satisfying the $\del\delbar$-lemma and admitting a non-vanishing $ABC$-Massey triple product.
\end{thm}

We start by considering the complex $3$-dimensional Heisenberg group $G:=\mathbb{H}(3,\C)$, i.e., the nilpotent group of matrices
\[
G=\left\{
\begin{pmatrix}
1 & z_1 & z_3\\
0 & 1   & z_2\\
0 & 0 & 1\
\end{pmatrix}: z_1,z_2,z_3\in\C \right\}.
\]
As an open set of $GL(n;\C)$, $G$ has standard holomorphic coordinates $\{z_1,z_2,z_3\}$.

If we take the lattice $\Gamma=G\cap GL(3;\Z[i])$, the compact quotient $M=\Gamma\backslash G$ is a complex nilmanifold of complex dimension $3$, the \emph{Iwasawa manifold}.

The group $G$ admits a left invariant coframe of $(1,0)$-forms
\[
\phi^1=dz_1,\quad \phi^2=dz_2,\quad \phi^3=dz_3-z_1dz_2
\]
which gives rise to a left-invariant integrable almost complex structure $J$ on $G$.

We note that the coframe $\{\phi^1,\phi^2,\phi^3\}$, and therefore the complex structure $J$, descends on the quotient $M$. Since the structure equations on $(M,J)$ are
\begin{equation}\label{eq:struct_iwa}
d\phi^1=0,\quad d\phi^2=0,\quad, d\phi^3=-\phi^{12},
\end{equation}
the complex structure $J$ is holomorphically parallelizable on $M$. Therefore, by \cite[Theorem 2.8]{A13}, we know that de Rham cohomology, Dolbeault cohomology, Bott-Chern cohomology and Aeppli cohomology of $(M,J)$ are isomorphic to the corresponding cohomologies of the Lie algebra $\mathfrak{g}$ of $G$ endowed with the complex structure $J$.

We point out that the Iwasawa manifold does not satisfy the $\del\delbar$-lemma. In fact, it is not formal \cite{Has}.

We now construct an orbifold of global-quotient-type starting from $M$. We first define the following action $\sigma\colon \C^3\rightarrow \C^3$ by
\begin{equation}
\sigma(z_1,z_2,z_3)=(iz_1,iz_2,-z_3), \quad\text{for} \quad (z_1,z_2,z_3)\in\C^3.
\end{equation}
We observe that as a group of biholomorphisms $\langle\sigma\rangle$ has finite order, since $\sigma^4=\id_{\C^3}$.

We need the following.
\begin{lem}\label{lem:sigma_def}
The action $\sigma$ is well defined on $M$.
\begin{proof}
We begin by noting that $G$ can be identified with $(\C^3,\star)$, where the product $\star$ is given by
\begin{equation}\label{eq:lem_w_d_0}
(z_1,z_2,z_3)\star(w_1,w_2,w_3)=(z_1+w_1,z_2+w_2,z_3+z_1w_2+w_3)
\end{equation}
for every $(z_1,z_2,z_3),(w_1,w_2,w_3)\in\C^3$.

We then need to show that, for $[z], [z']\in M$, if $[z]=[z']$, then $[\sigma(z)]=[\sigma(z')]$, or,  equivalently, that if $z=(z_1,z_2,z_3)\sim z'=(z_1',z_2',z_3')$, then $\sigma(z)\sim\sigma(z')$.

The equivalence is given by the action of multiplication on the left by elements of $\Gamma$, which, through the identification $G\simeq(\C^3,\star)$ reads $z\sim z'$ if, and only if, there exists $\gamma=(\gamma_1,\gamma_2,\gamma_3)\in(\Z[i])^3$ such that $z'=\gamma\star z$, which accounts to
\begin{align}\label{eq:lem_w_d_1}
\begin{cases}
z_1'=z_1+\gamma_1\\
z_2'=z_2+\gamma_2\\
z_3'=z_3+\gamma_1z_2+\gamma_3.
\end{cases}
\end{align}
Let us then assume that $z\sim z'$. We point out that
\begin{align*}
\sigma(z')=(iz_1',iz_2',-z_3')
\end{align*}
and, by (\ref{eq:lem_w_d_1}),
\begin{align}\label{eq:lem_w_d_2}
\sigma(z')=(iz_1+i\gamma_1,iz_2+i\gamma_2,-z_3-\gamma_1 z_2-\gamma_3).
\end{align}
Now choose $\tilde{\gamma}=(\tilde{\gamma}_1,\tilde{\gamma}_2,\tilde{\gamma}_3):=(i\gamma_1,i\gamma_2,-\gamma_3)\in\Z[i]$. By definition (\ref{eq:lem_w_d_0}) of the product $\star$ and equation (\ref{eq:lem_w_d_2}), it is easy to check that
\[
\sigma(z')=\tilde{\gamma}\star \sigma(z).
\]
\end{proof}
\end{lem}
As a consequence of Lemma \ref{lem:sigma_def}, we can define an action of $\sigma$ on $M$, given by $\sigma([z]):=[\sigma(z)]$, for every $[z]\in M$.

Let us now consider the quotient $M/\langle\sigma\rangle$. It is not a smooth manifold, as follows from the following lemma.
\begin{lem}\label{lem:fixed}
The action $\sigma$ on $M$ has 16 fixed points.
\begin{proof}
We need to find all the solution to the following equation
\begin{equation}\label{eq:fixed_00}
\sigma[z]=[z], \quad \text{for}\quad z=(z_1,z_2,z_3)\in\C^3 ,
\end{equation}
or, equivalently, to $\sigma(z)\sim z$, i.e., finding all the distinct solutions (up to equivalence) to
\begin{align}\label{eq:fixed_0}
\begin{cases}
iz_1=z_1+\gamma_1\\
iz_2=z_2+\gamma_2\\
-z_3=z_3+\gamma_1z_2+z_3,
\end{cases}
\end{align}
for $\gamma=(\gamma_1,\gamma_2,\gamma_3)\in(\Z[i])^3.$ Now, by writing $z_j=x_j+iy_j$ and $\gamma_j=m_j+ik_j$, the system (\ref{eq:fixed_0}) yields the following solutions
\begin{align}\label{eq:fixed_1}
\begin{cases}
z_1=\frac{1}{2}(-m_1+k_1+i(-m_1-k_1))\\
z_2=\frac{1}{2}(-m_2+k_2+i(-m_2-k_2))\\
z_3=\frac{1}{4}(m_1m_2-k_1k_2-m_1k_2-k_1m_2-2m_3+i(m_1m_2-k_1k_2+m_1k_2+k_1m_2-2k_3)).
\end{cases}
\end{align}
We observe that two points in $z=(z_1,z_2,z_3),z'=(z_1',z_2',z_3')\in\C^3$ satisfying (\ref{eq:fixed_1}) are equivalent in $(\Z[i])\backslash\C^3$ if, and only if, there exists $\lambda=(\lambda_1,\lambda_2,\lambda_3)\in(\Z[i])^3$ such that $z'=\lambda\star z$, i.e., 
\begin{align}\label{eq:fixed_2}
\begin{cases}
z_1'=z_1+\lambda_1\\
z_2'=z_2+\lambda_2\\
z_3'=z_3+\lambda_1z_2+\lambda_3.
\end{cases}
\end{align}
We look at the first equation. By writing each $\lambda_j=a_j+ib_j$ and using (\ref{eq:fixed_1}), we have that $z_1'-z_1=\lambda_1$ if, and only if,
\begin{align*}
&\frac{1}{2}(-m_1'+k_1'-(-m_1+k_1))=a_1\\
&\frac{1}{2}(-m_1'-k_1'-(m_1-k_1))=b_1.
\end{align*}
We notice that $[-m_1-k_1]=[-m_1+k_1]\in\frac{\Z}{2\Z}$. Therefore $z_1'-z_1=\lambda_1$ if and only if $[-m_1+k_1]=[-m_1'+k_1']\in\frac{\Z}{2\Z}$. By choosing as representatives for $\gamma_1$ as either $0$ or $1$, we obtain  that the distinct values of $z_1$, up to equivalence, are either $0$ or $\frac{1}{2}+\frac{i}{2}$. Analogously, this can be done for $z_2$, whose distinct values, up to equivalence, are $0$ or $\frac{1}{2}+\frac{i}{2}$. By plugging those values in the third equation of (\ref{eq:fixed_1}), we get that, in the case where $(z_1,z_2),(z_1',z_2')\neq(\frac{1}{2}+\frac{i}{2},\frac{1}{2}+\frac{i}{2})$, the third components of $z$ and $z'$ are, respectively, $z_3=-\frac{1}{2}m_3-\frac{i}{2}k_3$ and  $z_3'=-\frac{1}{2}m_3'-\frac{i}{2}k_3'$. Then, equation $z_3'-z_3=\lambda_1z_2+\lambda_3$ is satisfied  if and only if
\begin{align*}
&\frac{1}{2}(m_3-m_3')=a_3\\
&\frac{1}{2}(k_3-k_3')=b_3.
\end{align*}
Hence, by choosing $\gamma_3\in\{0,1,i,1+i\}$, we get that the only solutions, up to equivalence, are $z_3\in\{0,\frac{1}{2},\frac{i}{2},\frac{1}{2}+\frac{i}{2}\}$.

Finally, when $z_1=z_1'=z_2=z_2'=\frac{1}{2}+\frac{i}{2}$, we have expression for $z_3=\frac{1}{4}(1-2m_3+i(1-2k_3))$ and $z_3'=\frac{1}{4}(1-2m_3'+i(1-2k_3'))$ Therefore, $z_3'-z_3=\lambda_1z_2+\lambda_3$ holds if, and only if,
\begin{align*}
&\frac{1}{2}(m_3-m_3')=a_1+a_3\\
&\frac{1}{2}(k_3-k_3')=b_1+b_3.
\end{align*}
Thus, if one chooses $\gamma_3\in\{0,1,i,1+i\}$, one gets that the solutions, up to equivalence, are $z_3\in\{0,\frac{1}{2},\frac{i}{2},\frac{1}{2}+\frac{i}{2}\}$. By counting all the distinct solutions up to equivalence $z=(z_1,z_2,z_3)$ satisfying (\ref{eq:fixed_00}), i.e., the fixed point of $\sigma$ on $M$, we find that they are $16$ and, clearly, isolated.
\end{proof}
\end{lem}
As consequence of Lemma \ref{lem:fixed}, we obtain that $\hat{M}:=M/\langle\sigma\rangle$ is an orbifold of global-quotient-type.
Since 
\begin{equation*}
\sigma^*\phi^1=i\phi^1,\quad \sigma^*\phi^2=i\phi^2,\quad \sigma^*\phi^3=-\phi^3,
\end{equation*}
the complex of $\sigma$-invariant differential forms on $M$ is
\begin{equation*}
\textstyle\bigwedge^{\bullet,\bullet}\hat{M}=\Span_{\C}\left\langle1,\phi^{1\overline{1}},\phi^{1\overline{2}},\phi^{2\overline{1}},\phi^{2\overline{2}},\phi^{123},\phi^{12\overline{3}},\phi^{3\overline{12}},\phi^{\overline{123}},\phi^{12\overline{12}},\phi^{13\overline{13}},\phi^{13\overline{23}},\phi^{23\overline{13}},\phi^{23\overline{23}},\phi^{123\overline{123}}\right\rangle.
\end{equation*}
Let us fix $g$ the Hermitian metric on $\hat{M}$ with fundamental associated form $\omega=\frac{i}{2}(\phi^{1\overline{1}}+\phi^{2\overline{2}}+\phi^{3\overline{3}})$. We can now compute the cohomologies of $\hat{M}$ by definitions (\ref{eq:coom_orb}) and (\ref{eq:coom_orb1}) and via Theorems \ref{thm:coom_orb1} and \ref{thm:coom_orb2}. In particular, we prove the following.
\begin{lem}\label{lem:deldelbar_orb}
$\hat{M}$ satisfies the $\del\delbar$-lemma.
\begin{proof}
It suffices to the show that Fr\"olicher equality (\ref{eq:fr-orb}) holds and also $H_{\delbar}^{p,q}(\hat{M})\simeq H_{\delbar}^{q,p}(\hat{M})$ via complex conjugation. By easy computations of the harmonic representatives with respect to $g$, we see that the non-trivial de Rham cohomology spaces of $\hat{M}$ are
\begin{align*}
&H_{dR}^0(\hat{M};\C)=\Span_{\C}\langle 1\rangle\\
&H_{dR}^2(\hat{M};\C)=\Span_{\C}\langle\phi^{1\overline{1}},\phi^{1\overline{2}},\phi^{2\overline{1}},\phi^{2\overline{2}}\rangle\\
&H_{dR}^3(\hat{M};\C)=\Span_{\C}\langle\phi^{123},\phi^{\overline{123}}\rangle\\
&H_{dR}^4(\hat{M};\C)=\Span_{\C}\langle\phi^{13\overline{13}},\phi^{13\overline{23}},\phi^{23\overline{13}},\phi^{23\overline{23}}\rangle\\
&H_{dR}^6(\hat{M};\C)=\Span_{\C}\langle\phi^{123\overline{123}}\rangle,
\end{align*}
whereas  the non-trivial Dolbeault cohomology spaces of $\hat{M}$ are
\begin{align*}
&H_{\delbar}^{0,0}(\hat{M})=\Span_{\C}\langle 1\rangle\\
&H_{\delbar}^{1,1}(\hat{M})=\Span_{\C}\langle\phi^{1\overline{1}},\phi^{1\overline{2}},\phi^{2\overline{1}},\phi^{2\overline{2}}\rangle\\
&H_{\delbar}^{3,0}(\hat{M})=\Span_{\C}\langle\phi^{123}\rangle\\
&H_{\delbar}^{0,3}(\hat{M})=\Span_{\C}\langle\phi^{\overline{123}}\rangle\\
&H_{\delbar}^{2,2}(\hat{M})=\Span_{\C}\langle\phi^{13\overline{13}},\phi^{13\overline{23}},\phi^{23\overline{13}},\phi^{23\overline{23}}\rangle\\
&H_{\delbar}^{3,3}(\hat{M})=\Span_{\C}\langle\phi^{123\overline{123}}\rangle.
\end{align*}
By comparing the former and the latter spaces, we easily conclude the proof.
\end{proof}
\end{lem}
As a consequence, Bott-Chern and Aeppli cohomologies of $\hat{M}$ are immediately determined by $H_{BC}^{p,q}(\hat{M})=H_{\delbar}^{p,q}(\hat{M})$ and $H_A^{p,q}(\hat{M})\simeq\ast( H_{BC}^{3-p,3-q}(\hat{M}))$, yielding
\begin{align*}
&H_{BC}^{0,0}(\hat{M})=\Span_{\C}\langle 1\rangle & &H_A^{0,0}(\hat{M})=\Span_{\C}\langle 1\rangle\\
&H_{BC}^{1,1}(\hat{M})=\Span_{\C}\langle\phi^{1\overline{1}},\phi^{1\overline{2}},\phi^{2\overline{1}},\phi^{2\overline{2}}\rangle & &H_A^{1,1}(\hat{M})=\Span_{\C}\langle\phi^{1\overline{1}},\phi^{1\overline{2}},\phi^{2\overline{1}},\phi^{2\overline{2}}\rangle\\
&H_{BC}^{3,0}(\hat{M})=\Span_{\C}\langle\phi^{123}\rangle &&H_A^{3,0}(\hat{M})=\Span_{\C}\langle\phi^{123}\rangle\\
&H_{BC}^{0,3}(\hat{M})=\Span_{\C}\langle\phi^{\overline{123}}\rangle & &H_A^{0,3}(\hat{M})=\Span_{\C}\langle\phi^{\overline{123}}\rangle\\
&H_{BC}^{2,2}(\hat{M})=\Span_{\C}\langle\phi^{13\overline{13}},\phi^{13\overline{23}},\phi^{23\overline{13}},\phi^{23\overline{23}}\rangle & &H_A^{2,2}(\hat{M})=\Span_{\C}\langle\phi^{13\overline{13}},\phi^{13\overline{23}},\phi^{23\overline{13}},\phi^{23\overline{23}}\rangle\\
&H_{BC}^{3,3}(\hat{M})=\Span_{\C}\langle\phi^{123\overline{123}}\rangle & & H_A^{3,3}(\hat{M})=\Span_{\C}\langle\phi^{123\overline{123}}\rangle.
\end{align*}
We now define an $ABC$-Massey triple product on $\hat{M}$.
\begin{lem}\label{lem:ABC_prod_orb}
$\hat{M}$ admits a non vanishing $ABC$-Massey triple product.
\begin{proof}
Let us consider the following Bott-Chern cohomology classes
\[
[\alpha]:=[\phi^{1\overline{1}}]\in H_{BC}^{1,1}(\hat{M}),\quad [\beta]:=[\phi^{2\overline{2}}]\in H_{BC}^{1,1}(\hat{M}),\quad [\gamma]:=[\phi^{2\overline{2}}]\in H_{BC}^{1,1}(\hat{M}).
\]
We notice that, by structure equations (\ref{eq:struct_iwa}), we have that $\phi^{1\overline{1}}\wedge\phi^{2\overline{2}}=\del\delbar\phi^{3\overline{3}}$. Then, it is well-defined
\[
\langle[\alpha],[\beta],[\gamma]\rangle_{ABC}\in\frac{H_A^{2,2}(\hat{M})}{[\phi^{1\overline{1}}]_{BC}\cup H_A^{1,1}(\hat{M})+[\phi^{2\overline{2}}]_{BC}\cup H_A^{1,1}(\hat{M})},
\]
which, by Definition \ref{def:ABCM}, is represented by the non zero Aeppli cohomology class $[\phi^{23\overline{23}}]\in H_A^{2,2}(\hat{M})$. By the previous description of Aeppli cohomology, we note that the ideal $[\phi^{1\overline{1}}]_{BC}\cup H_A^{1,1}(\hat{M})+[\phi^{2\overline{2}}]_{BC}\cup H_A^{1,1}(\hat{M})$ is actually trivial in $H_A^{2,2}(\hat{M})$.

Hence, $\langle[\phi^{1\overline{1}}],[\phi^{2\overline{2}}],[\phi^{2\overline{2}}]\rangle_{ABC}$ is a non-vanishing $ABC$-Massey triple product on $\hat{M}$.
\end{proof}
\end{lem}
\begin{proof}[Proof of Theorem \ref{thm:ABC_deldelbar}]
(I) In view of Hironaka singularities resolutions theorem, see \cite{Hir}, it turns out that $\hat{M}$ admits a resolution. \newline
We will construct an explicit smooth resolution $\hat{M}$, proceedings as follows, see \cite{CG16}.
Define $\psi=\sigma^2$, i.e.,
\[
\psi(z_1,z_2,z_3)=(-z_1,-z_2,z_3)
\]
for every $(z_1,z_2,z_3)\in\C^3$. Clearly, $\psi$ descends to $M$  and has order 2 on $M$, i.e., since $\psi^2=\id_M$. The locus of fixed points by the action of $\psi$ on $M$ is the disjoint union of $8$ curves on $M$ given by
\begin{align*}
&\mathcal{C}_i=\{[z_1^0,z_2^0,z_3]:z_3\in\C\},
\end{align*}
with $(z_1^0,z_2^0)\in\{(0,\frac{1}{2}),(0,\frac{i}{2}),(0,\frac{1}{2}+\frac{i}{2}),(\frac{1}{2},0),(\frac{i}{2},0),(\frac{1}{2}+\frac{i}{2},0),(\frac{1}{2}+\frac{i}{2},\frac{1}{2}+\frac{i}{2})\}$.

Let us set $\mathcal{C}:=\mathcal{C}_1=\{[0,0,z_3]\}$. In a neighborhood  $U$ of $\mathcal{C}$ and local coordinates $(z_1,z_2,z_3)$, we write, locally, $\mathcal{C}=\{(z_1=0,z_2=0,z_3)\}$. We perform the blowup of $M$ along $\mathcal{C}$ by taking the set
\[
\tilde{U}=\{((z_1,z_2,z_3),[l_1:l_2]):z_1l_2-z_2l_1=0\}\subset U\times \mathbb{P}^2.
\]
Through the resulting the map $p\colon \mathit{Bl}_{\mathcal{C}}M\rightarrow M$, if $E:=p^{-1}(\mathcal{C})\simeq \mathbb{P}(\mathcal{N}_{\mathcal{C}/M})$ is the exceptional divisor, $\tilde{U}\setminus E$ projects bi-holomorphically onto $U\setminus\mathcal{C}$.

On $\tilde{U}_1=\{l_1\neq 0\}$, we have that $z_2=\frac{l_2}{l_2}z_1$ and local coordinates on $\tilde{U}_1$ are given by
\[
\zeta_1^{(1)}=z_1, \quad \zeta_2^{(1)}=\frac{l_2}{l_1},\quad \zeta_3^{(1)}=z_3,
\]
whereas on $\tilde{U}_2=\{l_2\neq 0\}$, we have that $z_1=\frac{l_1}{l_2}z_2$ and the following
\[
\zeta_1^{(2)}=\frac{l_1}{l_2},\quad \zeta_2^{(2)}=z_2,\quad \zeta_3^{(2)}=z_3
\]
are local coordinates on $\tilde{U}_2.$
In the following, we will show the procedure only on $\tilde{U}_1$ since on $\tilde{U}_2$ the approach is analogous.

Notice that $\psi$ induces a morphism $\tilde{\psi}$ on $\mathit{Bl}_{\mathcal{C}}M$. In particular, we have that, on $\tilde{U}_1$,
\[
\tilde{\psi}(\zeta_1^{(1)},\zeta_2^{(1)},\zeta_3^{(1)})=(-\zeta_1^{(1)},\zeta_2^{(1)},\zeta_3^{(1)}).
\]
Let us then consider the quotient $M'=\mathit{Bl}_{\mathcal{C}}M/\langle\tilde{\psi}\rangle$. On the quotient $\tilde{U}_1/\langle\tilde{\psi}\rangle\subset M'$,
the action $\sigma'$ induced by $\sigma$ acts as
\begin{equation}\label{eq:sigma_prime}
\sigma'([\zeta_1^{(1)},\zeta_2^{(1)},\zeta_3^{(1)}]_{\tilde{\psi}})=[i\zeta_1^{(1)},\zeta_2^{(1)},-\zeta_3^{(1)}]_{\tilde{\psi}}.\end{equation}
Note that, through local coordinates, $\tilde{U}_1/\langle\tilde{\psi}\rangle$
is identified with with $\C^3/\langle\tilde{\psi}\rangle$. So we construct local coordinates for the latter in the following way. The holomorphic map $f\colon\C^{3}\rightarrow \C^3$ defined by
\begin{gather*}
f(w_1,w_2,w_3)=(w_1^2,w_2,w_3), \quad \text{for}\,\, (w_1,w_2,w_3)\in\C^3,
\end{gather*}
which on local coordinates on $\tilde{U}_1$ acts as
\[
f(\zeta_1^{(1)},\zeta_2^{(1)},\zeta_3^{(1)})=((\zeta_1^{(1)})^2,\zeta_2^{(1)},\zeta_3^{(1)}),
\]
gives rise to the following diagram
\begin{equation*}
\begin{tikzcd}
\C^3_{(\zeta_1^{(1)},\zeta_2^{(1))},\zeta_3^{(1)})}\arrow[r, "f"]\arrow{d} & \C^3_{(w_1,w_2,w_3)} \\
\C^3/\langle\tilde{\psi}\rangle\arrow[ur, "\hat{f}",swap]
\end{tikzcd}
\end{equation*}
where $\hat{f}([\zeta^{(1)}]_{\tilde{\psi}}):=f(\zeta^{(1)})$ is well defined and, in fact, a biholomorphism.


Therefore, we can identify $\tilde{U}_1/\langle\tilde{\psi}\rangle$ with $\C^3_{(w_1,w_2,w_3)}$.
We now look for fixed point of $\sigma'$ on $M'$. Locally, we must then consider the action of $\sigma'$, which on $\C^3_{(w_1,w_2,w_3)}$ acts as $\tilde{\sigma}:=\hat{f}^{-1}\circ\sigma'\circ \hat{f}$. Recalling equation (\ref{eq:sigma_prime}), we see that
\begin{equation}\label{eq:tilde_sigma}
\tilde{\sigma}(w_1,w_2,w_3)=(-w_1,w_2,-w_3)
\end{equation}
for any $(w_1,w_2,w_3)\in\C^3$, yielding that the locus of fixed points of $\sigma'$ on $\tilde{U}_1/\langle\tilde{\psi}\rangle$ is given, locally, by the set
\[
\mathcal{D}=\{w_1=0,w_2,w_3=0\}.
\]
We now perform the further blowup $p'\colon \mathit{Bl}_{\mathcal{D}}M'\rightarrow M'$, by considering
\[
\tilde{\tilde{U}}^{(1)}=\{((w_1,w_2,w_3),[v_1:v_3]):w_1v_3-w_3v_1=0\}.
\] 
On $\tilde{\tilde{U}}_1^{(1)}:=\{v_1\neq 0\}$, we have that $w_3=\frac{v_3}{v_1}w_1$ and local coordinates are given by
\begin{equation}\label{eq:2_blow_up_coord_1}
\eta_1^{(1)}=w_1,\quad \eta_2^{(1)}=w_2,\quad \eta_3^{(1)}=\frac{v_3}{v_1},
\end{equation}
whereas on $\tilde{\tilde{U}}_3^{(1)}:=\{v_3\neq 0\}$, we have that $w_1=\frac{v_1}{v_3}w_3$ and the following
\begin{equation}\label{eq:2_blow_up_coord_2}
\eta_1^{(3)}=\frac{v_1}{v_3}, \quad \eta_2^{(3)}=w_2,\quad \eta_3^{(3)}=w_3,
\end{equation}
are local coordinates on $\tilde{\tilde{U}}_3^{(1)}$.

We now study the quotient $\mathit{Bl}_{\mathcal{D}}M'$ by the induced action of $\langle\tilde{\sigma}'\rangle$.
By recalling the local action of $\tilde{\sigma}$ (\ref{eq:tilde_sigma}) and the expressions (\ref{eq:2_blow_up_coord_1}) and (\ref{eq:2_blow_up_coord_2}) for local coordinates, on $\tilde{\tilde{U}}_1^{(1)}$ , we have that
\[
\tilde{\sigma}'(\eta_1^{(1)},\eta_2^{(1)},\eta_3^{(1)})=(-\eta_1^{(1)},\eta_2^{(1)},\eta_3^{(1)}),
\]
whereas on $\tilde{\tilde{U}}_3^{(1)}$, we have that
\[
\tilde{\sigma}'(\eta_1^{(3)},\eta_2^{(3)},\eta_3^{(3)})=(\eta_1^{(3)},\eta_2^{(3)},-\eta_3^{(3)}),
\]
i.e.,
\[
\tilde{\tilde{U}}_1^{(1)}/\langle\tilde{\sigma}'\rangle\simeq \frac{\C}{\pm\id}\times \C^2_{(\eta_2^{(1)},\eta_3^{(1)})}
\]
and
\[
\tilde{\tilde{U}}_3^{(1)}/\langle\tilde{\sigma}'\rangle\simeq \frac{\C}{\pm\id}\times \C^2_{(\eta_1^{(3)},\eta_2^{(3)})}.
\]
Hence, since $\tilde{\tilde{U}}_1^{(1)}/\langle\tilde{\sigma}'\rangle$ and $\tilde{\tilde{U}}_3^{(1)}/\langle\tilde{\sigma}'\rangle$ are smooth manifolds, the manifold $\mathit{Bl}_{\mathcal{D}}M'/\langle\tilde{\sigma}'\rangle$ is smooth.

As mentioned before, the same procedure can be applied starting from $\tilde{U}_2$, which results in finding
smooth resolutions of the singular points in the chart $\tilde{U}_2\subset\mathit{Bl}_{\mathcal{C}}M$.
\newline
Therefore, if we denote by $\tilde{M}_1$ the resulting complex manifold and the projection  $p''\colon \tilde{M}_1\rightarrow M/\langle\sigma\rangle$, we obtain a smooth resolution of the fixed curve $\mathcal{C}=\mathcal{C}_1$ on $M/\langle\sigma\rangle$.

By repeating the analogous procedure for every fixed locus $\mathcal{C}_i$, we obtain a smooth resolution
\[
\pi\colon \hat{M}\rightarrow \tilde{M},
\]
as the diagram summarizes
\[
\begin{tikzcd}
M \arrow{d} & \arrow{l} \mathit{Bl}_{Fix_\psi}M\arrow{d} &\\
M/\langle\psi\rangle \arrow{d} & M':=\mathit{Bl}_{Fix_{\psi}}M/\langle\tilde{\psi}\rangle \arrow{l} \arrow{d} & \mathit{Bl}_{Fix_{\tilde{\sigma}}}M' \arrow{l}\arrow{d}\\
\hat{M}:=M/\langle\sigma\rangle & M'/\langle\sigma'\rangle \arrow{l} & \mathit{Bl}_{Fix_{\tilde{\sigma}}}M'/\langle\tilde{\sigma}'\rangle\arrow{l}.
\end{tikzcd}
\]

(II) We now show the following:\vskip.2truecm\noindent
$(i)$ $\tilde{M}$ admits a non-vanishing $ABC$-Massey triple
product; \vskip.1truecm\noindent
$(ii)$ $\tilde{M}$ satisfies the $\del\delbar$-lemma.
\vskip.3truecm\noindent
$(i)$ We proceed by considering the pull-back through $\pi$ of the Bott-Chern cohomology classes used in Lemma \ref{lem:ABC_prod_orb}, i.e., we consider the classes $[\pi^*\phi^{1\overline{1}}]\in H_{BC}^{1,1}(\tilde{M})$ and $[\pi^*\phi^{2\overline{2}}]\in H_{BC}^{1,1}(\tilde{M})$. They are well-defined and non-vanishing, by Theorem \ref{thm:pullback_orb}. Since $\pi^*(\phi^{1\overline{1}})\wedge \pi^*(\phi^{2\overline{2}})=\del\delbar (\pi^*\phi^{3\overline{3}})$, the $ABC$-Massey product
\[
\langle[\pi^*\phi^{1\overline{1}})],[\pi^*\phi^{2\overline{2}}],[\pi^*\phi^{2\overline{2}}] \rangle_{ABC}\in\frac{H_A^{2,2}(\tilde{M})}{[\pi^*\phi^{1\overline{1}}]_{BC}\cup H_A^{1,1}(\tilde{M})+[\pi^*\phi^{2\overline{2}}]_{BC}\cup H_A^{1,1}(\tilde{M})}
\]
is well-defined and represented by $[\pi^*\phi^{23\overline{23}}]\in H_A^{2,2}(\tilde{M})$. Again, by Theorem \ref{thm:pullback_orb}, this class is not vanishing.

It remains to show that 
$$[\pi^*\phi^{23\overline{23}}]_A\notin [\pi^*\phi^{1\overline{1}}]_{BC}\cup H_A^{1,1}(\tilde{M})+[\pi^*\phi^{2\overline{2}}]_{BC}\cup H_A^{1,1}(\tilde{M}).$$ 
By contradiction, let us suppose the converse, i.e.,
\begin{equation}\label{eq:ABC_resolution}
[\pi^*\phi^{23\overline{23}}]_A=[\pi^*\phi^{1\overline{1}}]_{BC}\cup [F]_A+[\pi^*\phi^{2\overline{2}}]_{BC}\cup [G]_A,
\end{equation}
for some $[F],[G]\in H_A^{1,1}(\tilde{M})$. Let us now multiply by $[\pi^*\phi^{1\overline{1}}]_{BC}$ each side of (\ref{eq:ABC_resolution}), to obtain
\begin{align*}
[\pi^*\phi^{123\overline{123}}]_A&=[\pi^*\phi^{1\overline{1}}\wedge\pi^*\phi^{2\overline{2}}]_{BC}\cup[G]_A\\
&=[\pi^*(\del\delbar\phi^{3\overline{3}})]_{BC}\cup[G]_A\\
&=[\del\delbar(\pi^*\phi^{3\overline{3}})]_{BC}\cup[G]_A\\
&=[\del\delbar(\pi^*\phi^{3\overline{3}}\wedge G)]_A=0\in H_{A}^{2,2}(\tilde{M}),
\end{align*}
which leads to contradiction, since $\pi^*$ is injective by Theorem \ref{thm:pullback_orb} and $[\phi^{123\overline{123}}]_A\neq0.$
\vskip.2truecm\noindent
$(ii)$ We now observe that the fixed points loci along which we perform the blowups are complex lines, which are naturally K\"ahler. Therefore, they satisfy the $\del\delbar$-lemma. As proved in Lemma \ref{lem:deldelbar_orb}, also $\hat{M}$ satisfies the $\del\delbar$-lemma. We can then apply \cite[Theorem 25]{ASTT}, to conclude that the resolution $\tilde{M}$ of $\hat{M}$ satisfies the $\del\delbar$-lemma.
\end{proof}

\begin{equation}\label{tab:BC_cohom}
\begin{array}{ll}
\toprule
(p,q) & H_{BC}^{p,q}(M,J_t), \quad t\in\Delta\setminus0 \\
\bottomrule
(0,0) & \C\langle1\rangle \\
\midrule
(1,0) &  \C\langle \eta_t^1 \rangle \\
(0,1) & \C\langle \eta_t^{ \overline{1}}\rangle \\
\midrule
(2,0) & \C\langle \eta_t^{23}\rangle\\
(1,1) & \C\langle \eta_t^{1\overline{1}}, e^{z_1-\overline{z}_1} \eta_t^{2\overline{3}}, e^{\overline{z}_1-z_1}\eta_t^{3\overline{2}}\rangle\\
(0,2) & \C\langle \eta_t^{\overline{23}}\rangle\\
\midrule
(3,0) & \C\langle \eta_t^{123}\rangle\\
(2,1) & \C\langle e^{z_1-\overline{z}_1}\eta_t^{12\overline{3}}, e^{\overline{z}_1-z_1}\eta_t^{13\overline{2}},\eta_t^{23\overline{1}} \rangle\\
(1,2) & \C\langle e^{\overline{z}_1-z_1}\eta_t^{3\overline{12}}, e^{z_1-\overline{z}_1}\eta_t^{2\overline{13}}, \eta_t^{1\overline{23}}\rangle\\
(0,3) & \C\langle \eta_t^{\overline{123}}\rangle\\
\midrule
(3,1) & \C\langle \eta_t^{123\overline{1}}\rangle\\
(2,2) & \C\langle e^{z_1-\overline{z}_1}\eta_t^{12\overline{12}}, e^{\overline{z}_1-z_1}\eta_t^{13\overline{12}}, \eta_t^{23\overline{23}}\rangle\\
(1,3) & \C\langle \eta_t^{1\overline{123}}\rangle\\
\midrule
(3,2) & \C\langle \eta_t^{123\overline{23}}\rangle\\
(2,3) & \C\langle\eta_t^{23\overline{123}}\rangle\\
\midrule
(3,3) & \C\langle\eta_t^{123\overline{123}}\rangle\\
\bottomrule
\end{array}
\end{equation}
\begin{equation}\label{tab:BC_cohom_nak}
\begin{array}{ll}
\toprule
(p,q) & H_{BC}^{p,q}(M,J) \\
\bottomrule
(0,0) & \C\langle1\rangle \\
\midrule
(1,0) &  \C\langle \eta^1 \rangle \\
(0,1) & \C\langle \eta^{\overline{1}}\rangle \\
\midrule
(2,0) & \C\langle \eta^{12},\eta^{13},\eta^{23}\rangle\\
(1,1) & \C\langle \eta^{1\overline{1}}, e^{\overline{z}_1-z_1} \eta^{1\overline{2}}, e^{z_1-\overline{z}_1}\eta^{1\overline{3}}, e^{z_1-\overline{z}_1}\eta^{2\overline{1}}, \eta^{z_1-\overline{z}_1} \eta^{2\overline{3}}, e^{\overline{z}_1-z_1}\eta^{3\overline{1}}, e^{\overline{z}_1-z_1}\eta^{3\overline{2}}\rangle\\
(0,2) & \C\langle \eta^{\overline{12}},\eta^{\overline{13}}, \eta^{\overline{23}}\rangle\\
\midrule
(3,0) & \C\langle \eta^{123}\rangle\\
(2,1) & \C\langle \eta^{12\overline{1}}, e^{z_1-\overline{z}_1}\eta^{12\overline{1}}, e^{\overline{z}_1-z_1}\eta^{12\overline{2}}, e^{z_1-\overline{z}_1}\eta^{12\overline{3}}, \eta^{13\overline{1}}, e^{\overline{z}_1-z_1}\eta^{13\overline{1}}, e^{\overline{z}_1-z_1}\eta^{13\overline{2}}, e^{z_1-\overline{z}_1}\eta^{13\overline{3}},\eta^{23\overline{1}} \rangle\\
(1,2) & \C\langle \eta^{1\overline{12}}, e^{\overline{z}_1-z_1}\eta^{1\overline{12}}, \eta^{1\overline{13}}, e^{z_1-\overline{z}_1}\eta^{1\overline{13}}, \eta^{1\overline{23}}, e^{z_1-\overline{z}_1}\eta^{2\overline{12}}, e^{z_1-\overline{z}_1}\eta^{2\overline{13}}, e^{\overline{z}_1-z_1}\eta^{3\overline{12}}, e^{\overline{z}_1-z_1}\eta^{3\overline{13}}\rangle\\
(0,3) & \C\langle \eta^{\overline{123}}\rangle\\
\midrule
(3,1) & \C\langle \eta^{123\overline{1}}, e^{\overline{z}_1-z_1}\eta^{123\overline{2}}, e^{z_1-\overline{z}_1}\eta^{123\overline{3}}\rangle\\
(2,2) & \C\langle e^{z_1-\overline{z}_1}\eta^{12\overline{12}}, e^{\overline{z}_1-z_1}\eta^{13\overline{12}}, \eta^{23\overline{23}}\rangle\\
(1,3) & \C\langle \eta^{1\overline{123}}, e^{z_1-\overline{z}_1}\eta^{2\overline{123}}, e^{\overline{z}_1-z_1}\eta^{3\overline{123}}\rangle\\
\midrule
(3,2) & \C\langle \eta^{123\overline{12}}, e^{\overline{z}_1-z_1}\eta^{123\overline{12}}, \eta^{123\overline{13}}, e^{z_1-\overline{z}_1} \eta^{123\overline{13}},\eta^{123\overline{23}}\rangle\\
(2,3) & \C\langle \eta^{12\overline{123}},e^{z_1-\overline{z}_1}\eta^{12\overline{123}},\eta^{13\overline{123}},e^{\overline{z}_1-z_1}\eta^{13\overline{123}}, \eta^{23\overline{123}}\rangle\\
\midrule
(3,3) & \C\langle\eta^{123\overline{123}}\rangle\\
\bottomrule
\end{array}
\end{equation}

\end{document}